\documentclass[11pt]{amsart}
\usepackage[margin=1in]{geometry}
\usepackage{amsmath,amssymb,amsthm,mathtools,mathrsfs,booktabs,longtable,array}
\usepackage{graphicx,colortbl}
\newcolumntype{L}[1]{>{\raggedright\arraybackslash}p{#1}}
\usepackage[T1]{fontenc}
\usepackage{lmodern,microtype,xcolor}
\definecolor{darkblue}{rgb}{0,0,0.7}
\usepackage[colorlinks=true,linkcolor=darkblue,citecolor=darkblue,urlcolor=darkblue]{hyperref}
\usepackage{thm-restate}
\RemoveFromHook{package/cleveref/after}[firstaid/aliascounter]
\usepackage[capitalize,nameinlink]{cleveref}
\newcommand{\defeq}{\coloneqq}
\newcommand{\operatorstyle}{\textstyle}
\DeclareMathOperator{\Cat}{Cat}
\DeclareMathOperator{\rdet}{rdet}
\newcommand{\defn}[1]{\emph{\color{darkblue}#1}}
\newtheorem{theorem}{Theorem}[section]
\newtheorem{proposition}[theorem]{Proposition}
\newtheorem{lemma}[theorem]{Lemma}
\newtheorem{corollary}[theorem]{Corollary}
\theoremstyle{definition}
\newtheorem{definition}[theorem]{Definition}
\newtheorem{example}[theorem]{Example}
\newtheorem{remark}[theorem]{Remark}
\newtheorem{problem}[theorem]{Problem}
\crefname{appendix}{Appendix}{Appendices}
\Crefname{appendix}{Appendix}{Appendices}

\AddToHook{cmd/appendix/before}{%
  \crefalias{section}{appendix}%
}
\usepackage[cal=boondoxo]{mathalfa}
\usepackage[colorinlistoftodos]{todonotes}
\newcommand{\C}{\mathbb C}\newcommand{\Q}{\mathbb Q}
\newcommand{\Z}{\mathbb Z}\newcommand{\N}{\mathbb Z_{\ge0}}

\DeclareMathOperator{\Sym}{Sym}\DeclareMathOperator{\Hom}{Hom}
\DeclareMathOperator{\Hilb}{Hilb}
\DeclareMathOperator{\Fix}{Fix}

\title{Posh parking spaces}
\author{Jason Stack}
\author{Nathan Williams}
\hypersetup{pdftitle={Posh parking spaces},
  pdfauthor={Jason Stack and Nathan Williams}}
\date{September 28, 2026}
\begin{document}
\begin{abstract}
Let $W$ be an
irreducible complex reflection group with reflection representation $V$.  A $W$-stable, faithful homogeneous system of
parameters $\Theta \subseteq \mathrm{Sym}(V^*)$ of common positive degree $p$ is called a \emph{posh hsop}; $\Theta$ \emph{carries} a $W$-representation $U$ if $\Theta \simeq U$ as ungraded $W$-modules. We classify the \emph{posh} pairs $(p,U)$ for which a posh hsop of degree $p$ carrying $U$ exists. Every such $U$ is a Galois twist of $V^*$. Extending work of Ito and Okada, we deduce that the quotient $S/(\Theta)$ is a permutation module for $W$ if and only if $U \simeq V^*$.
\end{abstract}
\maketitle
\section{Introduction}

Let $W$ be a finite complex reflection group acting on a vector space
$V$ of rank $r$ (the \defn{reflection representation}), and put $S\defeq \Sym(V^*)$ with graded homogeneous components $S=\bigoplus_{p=0}^\infty S_p$.  The \defn{field of definition} of $W$ is
\[
 K_W\defeq \mathbb{Q}\bigl(\chi_V(w):w\in W\bigr).
\]
By a theorem of M.~Benard and D.~Bessis~\cite{Benard,BessisField}, every irreducible complex
$W$-representation can be realized over $K_W$; see I.~Marin and J.~Michel~\cite[Introduction]{MarinMichel}.
Write $\mathfrak f_W$ for the \defn{conductor} of this field, the least
positive integer such that $K_W\subseteq\mathbb{Q}(\zeta_{\mathfrak f_W})$,
where $\zeta_{\mathfrak f_W}\defeq\exp(2\pi i/\mathfrak f_W)$.  (For irreducible well-generated $W$, we have $\mathfrak f_W$ divides the \defn{Coxeter number}\footnote{For an irreducible well-generated group, the Coxeter number is $h=(N+N^*)/r$, where $N$ is the number of reflections and $N^*$ is the number of reflecting hyperplanes. Using the same formula for $h$ when $W$ is not well-generated, $\mathfrak f_W\mid h$ exactly when $W=G(2s,2,r)$ with $s>1$ odd and $r\geq2$.} $h$.)  For $\gcd(k,\mathfrak f_W)=1$, let $\sigma_k$ be the automorphism
$\zeta_{\mathfrak f_W}\mapsto\zeta_{\mathfrak f_W}^k$.
The \defn{Galois twist} $(V^*)^{\sigma_k}$ is obtained by applying
$\sigma_k$ entrywise to a matrix realization of $V^*$ over $\mathbb{Q}(\zeta_{\mathfrak f_W})$ (note, however, that $\sigma_k$ can become trivial when restricted to $K_W$---for example, for real groups, $V\simeq V^*$ so that complex conjugation is a trivial Galois twist).

\subsection{Posh hsops and pairs}

Following D.~Armstrong, V.~Reiner, and B.~Rhoades~\cite{ARR}, and
Rhoades~\cite{RhoadesFuss}, we consider the following parameter spaces.
\begin{definition}
For a positive integer $p$, a $W$-stable \defn{homogeneous system of
parameters (hsop)} of degree $p$ is an $r$-dimensional $W$-stable
subspace $\Theta\subseteq S_p$ whose common zero is the origin.\footnote{Equivalently, $S/(\Theta)$ is finite-dimensional.}
It \defn{carries} $U$ if $\Theta\cong U$ as an ungraded $W$-module.
If the action of $W$ on $\Theta$ is faithful, we call $\Theta$
\defn{posh}. We call a pair $(p,U)$ \defn{posh} if a
posh hsop of degree $p$ carrying $U$ exists.
\label{def:parking_hsop}
\label{def:posh}
\end{definition}

\begin{example}\label{ex:intro-cyclic}
For $m>1$, let the cyclic group $W=C_m$ act on $V=\C$ by multiplication by
$\zeta_m$, so that it acts on $x\in V^*$ by $\zeta_m^{-1}$.
For every $p>0$, $\Theta=\C x^p$ is a $W$-stable hsop.
It is posh exactly when $\gcd(p,m)=1$, so that $(p,(V^*)^{\sigma_p})$ is posh.
\end{example}

In fact, posh hsops always carry some Galois twist of the reflection representation.
\begin{restatable}{lemma}{faithfulhsoplemma}\label{a:cor:faithful-hsop}
Let $W\subseteq\operatorname{GL}(V)$ be a nontrivial irreducible complex reflection
group of rank $r$, and let $\Theta\subseteq S_p$ be a $W$-stable hsop
carrying $U$. Then
\[
 \Theta\text{ is posh}
 \quad\Longleftrightarrow\quad
 U\cong(V^*)^\sigma\text{ for some }\sigma.
\]
\end{restatable}

We caution that non-isomorphic twists can form posh pairs with the same degree.

\begin{example}\label{ex:intro-h3}
The Coxeter group $H_3$ is defined over $K_{H_3}=\mathbb{Q}(\sqrt{5})\subseteq\mathbb{Q}(\zeta_5)$, and
$\mathfrak f_{H_3}=5$.  As a real group, complex conjugation ($\sigma_4$) acts trivially to preserve the reflection representation, whereas the nontrivial Galois twist $\sigma: \sqrt{5}\mapsto-\sqrt{5}$ is realized by both $\sigma_2$ and $\sigma_3$. Both $(5,V^*)$ and $(5,(V^*)^\sigma)$ are posh. (For more information, see~\cite{BalagovicPuranik} and~\Cref{c:sec:seeds}.)
\end{example}

Several families of posh pairs were previously known:
\begin{itemize}
\item For $S_n$, M.~Haiman proved that $(p,V^*)$ is posh exactly when
$\gcd(p,n)=1$, by composing equivariant polynomial maps~\cite{Haiman}.
We review his construction in~\cref{c:sec:cherednik}.
\item For Weyl groups and $\gcd(p,h)=1$, A.~Oblomkov and Z.~Yun give
a geometric realization of the posh pairs $(p,V^*)$ using homogeneous affine Springer fibers~\cite{OblomkovYun}.
\item The rational Cherednik approach builds on I.~Gordon's work on
diagonal-coinvariant quotients~\cite{GordonDiagonal} and on
Y.~Berest--P.~Etingof--V.~Ginzburg's finite-dimensional polynomial
representations and character formulas~\cite{BEGFinite}.
For irreducible well-generated complex reflection groups and $\gcd(p,h)=1$,
I.~Gordon and S.~Griffeth use \emph{rational Cherednik algebras} to prove
that $(p,(V^*)^{\sigma_p})$ is posh~\cite[Theorem 2.11 and Section 2.18]{GG}.
We review their construction in~\cref{c:sec:cherednik}; for P.~Etingof's
reducedness argument in the real case, see~\cite[Section~12]{ARR}. 
\end{itemize}

Our first result classifies all posh pairs.  In light of the constructions above, we first give the easily-stated (and expected) specialization of our classification to well-generated complex reflection groups.

\begin{restatable}{proposition}{coprimeproposition}\label{prop:coprime-hsop}
Let $W$ be a well-generated finite irreducible complex reflection group,
let $p>0$ be coprime to its Coxeter number $h$, and let $U$ be a
$W$-module. Then $(p,U)$ is posh iff $U\simeq(V^*)^{\sigma_p}$.
\end{restatable}

Note that~\Cref{prop:coprime-hsop} does not apply to the two non-isomorphic posh hsops of degree $5$ for $H_3$ from~\Cref{ex:intro-h3}, since $5$ is not coprime to the Coxeter number $h=10$.  We now state a complete classification of the degrees $p$ and representations $U$ for which there exists a posh hsop of degree $p$ carrying $U$ (necessarily a Galois twist $(V^*)^\sigma$ by~\Cref{a:cor:faithful-hsop}).  In contrast to~\Cref{prop:coprime-hsop}, the particular twist $\sigma$ is not necessarily specified by the degree $p$.

\begin{restatable}{theorem}{classificationtheorem}\label{c:thm:main}
Let $W$ be a nontrivial irreducible finite complex reflection group of rank $r$,
fix a positive integer $p$, and fix $U$ a faithful $r$-dimensional
representation of $W$. Write $a_i\defeq e_i(U)$ and
$b_i\defeq e_i(U^*)$ for the fake-degree exponents, and let
$d_1,\ldots,d_r$ be the invariant degrees. 

The following are
equivalent:
\begin{enumerate}
\item The pair $(p,U)$ is posh.
\item The class function
\[
 w\longmapsto\frac{\det(1-q^p w\mid U)}{\det(1-qw\mid V^*)}
\]
is the character of a finite-dimensional graded $W$-module.
\item Every reflection of $W$ acts on $U$ as a reflection, and both
\begin{equation*}
 \operatorstyle
 \Cat^+_{p,U}(q)\defeq\prod\limits_{i=1}^r\frac{1-q^{p+b_i}}{1-q^{d_i}}
 \qquad\text{and}\qquad
 \Cat^-_{p,U}(q)\defeq\prod\limits_{i=1}^r
                              \frac{q^{a_i}-q^p}{1-q^{d_i}}
\end{equation*}
are polynomials.
\item The pair $(p,U)$ appears in~\cref{tab:infinite-algebraic} (infinite families) or~\Cref{c:tab:exceptional} (exceptional types).
\end{enumerate}
\end{restatable}


The Koszul character formula of~\cref{prop:koszul} proves
$(1)\Rightarrow(2)$; see D.~Bessis and V.~Reiner~\cite[Proposition~4.2 and its proof]{BR}.
For $(2)\Rightarrow(3)$, polynomiality of the determinant quotient
at each reflection forces it to act on $U$ as a reflection.
A theorem of Marin--Michel~\cite[Theorem~1.2]{MarinMichel}, recalled
in~\cref{b:thm:MarinMichel}, classifies faithful irreducible reflection
representations as Galois twists. Its extension to reducible modules
in~\cref{a:prop:faithful-reflection-module} then identifies $U$ as a
Galois twist of $V^*$, as shown in~\cref{c:prop:reduction}.
The exterior-invariant theorem of P.~Orlik and L.~Solomon~\cite[Theorem~3.1]{OS},
recalled in~\cref{b:thm:OS}, identifies the trivial and determinant
multiplicities with $\Cat^+_{p,U}$ and $\Cat^-_{p,U}$
in~\cref{cor:formal-products}. For an actual hsop, these are the
Hilbert series of $A_\Theta^W$ and $\Hom_W(\det\Theta,A_\Theta)$
in~\cref{a:thm:automatic}.
We then identify all pairs satisfying (3) and prove that each is posh
by constructing an hsop.

For the infinite families, coordinate powers supply both the hsops
and the fake-degree calculation used to prove necessity
in~\cref{a:prop:monomial,c:prop:infinite-classification}; Haiman's
construction in~\cref{c:ex:haiman} handles type $A$.
For the exceptional groups, a finite cyclotomic test lists the
candidates in~\cref{c:prop:exceptional-polynomiality}.
Gordon--Griffeth's construction covers the well-generated coprime
cases in~\cref{c:thm:GG}. The remaining well-generated cases reduce
by composition to eight seed hsops for $G_4,G_6,H_3,G_{27}$;
we prove this coverage in~\cref{c:prop:seed-coverage} and verify
the seed maps in~\cref{c:lem:seeds}.
Scalar extensions and restriction from $E_8$ to $G_{31}$ complete
the non-well-generated cases in~\cref{c:lem:scalar,c:lem:eigenspace}.

Two operations recur in our constructions:
\begin{itemize}
\item composition of posh hsops multiplies their common degrees and
composes their Galois twists in~\cref{c:lem:composition}; and
\item pairing an hsop carrying $U$ with basic $U$-valued covariants
produces invariant parameters of degrees $p+e_i(U^*)$;
for $U=V^*$, these are obtained by contracting the hsop map with
the differentials of basic invariants in~\cref{b:prop:invariant-parameters}.
\end{itemize}





\subsection{Posh parking spaces}
We now turn to the combinatorics of the (finite-dimensional) quotient of $S$ by a posh hsop.
\begin{definition}
Given an hsop $\Theta\subseteq S_p$, define
\[
 A_\Theta\defeq S/(\Theta).
\]
Here $(\Theta)$ denotes the (homogeneous and $W$-stable) ideal of $S$ generated by the elements
of $\Theta$.  
When $\Theta$ is posh, we call $A_\Theta$ the corresponding
\defn{posh parking space} with parameter $p$, or simply a \defn{parking space}.
\label{def:parking_space}
\end{definition}
This quotient has vector-space dimension $p^r$ and has the
structure of a graded $W$-module; see~\cref{prop:koszul}.
We say that a posh hsop is \defn{untwisted} when it
carries $V^*$; in this case the ungraded character of the parking space is
$w\mapsto p^{\dim V^w}$ by~\cref{eq:twisted-character}.  In particular, for Weyl groups at all degrees $p$, and for well-generated complex reflection groups at the Fuss degrees $p=mh+1$, the posh hsop is untwisted.

When $A_\Theta$ happens to be a permutation module, it is interesting to construct combinatorial models for the corresponding ungraded permutation module (similarly to how the group algebra of $W$ is an ungraded combinatorial model of the coinvariant algebra; see~\Cref{b:sec:coinvariants,prob:nonfaithful-hsops}).  Such combinatorial $p$-parking modules have been constructed at many levels of generality:

\begin{itemize}
\item For the symmetric group $S_n$, P.~H.~Edelman~\cite{Edelman} gave combinatorial evidence that noncrossing partitions could be used to construct $(h+1)$-parking spaces.  
  D.~Armstrong--N.~A.~Loehr--G.~S.~Warrington constructed rational parking functions with character $w\mapsto p^{\dim V^w}$ whenever $\gcd(n,p)=1$~\cite[Theorem~3 and Corollary~4]{ALWRational}.  
\item Haiman showed that, after forgetting the grading, $p$-parking spaces for $S_n$ have the permutation character of the finite torus $Q/pQ$~\cite{Haiman}.  For \emph{very good}
parameters of a Weyl group, E.~Sommers~\cite[Proposition 3.9]{Sommers} showed
that $Q/pQ$ has character $w\mapsto p^{\dim V^w}$.
\item For Coxeter groups, D.~Armstrong--V.~Reiner--B.~Rhoades~\cite{ARR} and
Rhoades~\cite{RhoadesFuss} constructed noncrossing (and nonnesting) $p$-parking spaces for $p=h+1$ and $p=mh+1$, respectively. 
\item T.~Douvropoulos and M.~Josuat-Verg\`es completed the character
formula for noncrossing $(mh+1)$-parking spaces in types $E_7$ and $E_8$,
and established its analogue for prime noncrossing $(mh-1)$-parking
spaces for irreducible Coxeter groups~\cite[Theorems~3.10 and~3.12]{DJVRefined}.
\item J.~Stack~\cite{StackParking} generalized these constructions for noncrossing $(mh+1)$-parking spaces to well-generated complex reflection groups.
\end{itemize}

Using character theory, Y.~Ito and S.~Okada~\cite[Theorem~1.4]{IO} classified exactly when the untwisted determinant quotient is the character of a graded $W$-module whose underlying ungraded module is a permutation representation.  
Our second main result resolves a question of Ito--Okada~\cite[Problem 7.3(2)]{IO}, since~\Cref{c:thm:main} specializes to an algebraic construction of these graded modules as posh parking spaces. 

\begin{restatable}{theorem}{permutationtheorem}\label{c:thm:untwisted}
Let $W$ be a finite irreducible complex reflection group and $\Theta$ a posh hsop. Then $A_\Theta$ is a permutation posh parking space iff $\Theta$ carries $V^*$.
\end{restatable}

The classification is given in~\cref{tab:infinite-untwisted,tab:exceptional-untwisted}.
For non-dihedral $W$, these are also exactly the degrees in which $w\mapsto p^{\dim V^w}$ is a permutation character.  For the non-dihedral members of the infinite family $G(m,d,r)$, an untwisted parking space with parameter $p$ exists iff $p\equiv1\pmod m$.

Just as in~\Cref{prop:coprime-hsop}, for well-generated groups and degrees coprime to the Coxeter number,
we obtain an easily-stated specialization.

\begin{restatable}{corollary}{coxeterpermutationcorollary}\label{cor:coxeter-permutation}
Let $W$ be an irreducible well-generated complex reflection group with
Coxeter element\footnote{By \emph{Coxeter element}, we mean a
$\zeta_h$-regular element, where $\zeta_h=\exp(2\pi i/h)$;
thus $cv=\zeta_hv$ for some $v$ outside every reflecting hyperplane.
See~\cite[Section~1.1]{RRS}.} $c$ of order $h$, and let $\Theta$ be a posh hsop of degree $p$
with $\gcd(p,h)=1$. Write $e_i\defeq d_i-1$ for the exponents of $W$.

The following are equivalent:
\begin{itemize}
\item $\Theta$ carries $V^*$;
\item $A_\Theta$ is a permutation posh parking space;
\item $c^p$ is conjugate to $c$ in $W$; and
\item $
 \{p e_i\bmod h\}_{i=1}^r=\{e_i\}_{i=1}^r$
as multisets, where residues are taken in $\{0,\ldots,h-1\}$.
\end{itemize}
\end{restatable}

\begin{remark}
P.~Galashin--T.~Lam--M.-T.~Trinh--N.~Williams constructed rational noncrossing parking objects for finite Coxeter groups and proved that there are $p^r$ such objects when $\gcd(p,h)=1$~\cite[Theorem~1.8]{GLTWRational}.  Realizing these objects as an ungraded parking module also requires a $W$-action with the desired character.
For well-generated groups and $p$ coprime to $h$, we plan to construct combinatorial models for ungraded permutation posh parking spaces, generalizing our work in~\cite{DefantWilliams} where we used Springer's theorem that $c^p$ and $c$ are always conjugate in a Weyl group to construct and explain reflection subword models for rational Catalan objects.
\end{remark}

\subsection{Outline}
We give background in \cref{sec:background}. In \cref{sec:generators}
we establish the generator restriction and multiplicity formulas, and
prove the permutation criterion~\cref{c:thm:untwisted} using a
specialization of the character formula~\cref{eq:twisted-character}. The classification proof
in \cref{sec:proof-algebraic} proceeds through the infinite families,
the exceptional polynomiality test, the seed constructions, and the
restriction constructions, then assembles these results and deduces
the coprime specialization. In \cref{sec:questions} we ask which
nonfaithful parameter modules occur and discuss constructions with
generators of different degrees. The full classification tables
appear in \cref{sec:tables}.


\section{Background}\label{sec:background}

Recall that $S\defeq\C[V]=\Sym(V^*)$ is the polynomial ring on $V$.
In this section, we use the coinvariant algebra to locate copies of a
prescribed $W$-module in $S$, and explain how an hsop determines the
character of its quotient. We then collect the representation-theoretic
results and earlier constructions used to prove \cref{c:thm:main}.

\subsection{Coinvariants and parameter spaces}\label{b:sec:coinvariants}


All representations are over $\C$. A \defn{complex reflection} is a
finite-order linear transformation whose fixed space is a hyperplane. A
\defn{complex reflection group} is a finite subgroup of
$\operatorname{GL}(V)$ generated by complex reflections. A rank-$r$
group is \defn{well-generated} if it can be generated by $r$ reflections.
We let $W$ act on $S$ by
$(wf)(x)\defeq f(w^{-1}x)$, and write
$[p]_q\defeq1+q+\cdots+q^{p-1}$.

The invariant ring is a polynomial algebra
$S^W=\C[f_1,\ldots,f_r]$ on homogeneous \defn{basic invariants}; this is
the theorem of G.~C.~Shephard and J.~A.~Todd~\cite{ST} and C.~Chevalley~\cite{Chevalley}.
Their degrees $d_i\defeq\deg f_i$ are the \defn{invariant degrees}.
The \defn{coinvariant algebra} is
\[
 C_W\defeq S/(f_1,\ldots,f_r),
\]
where $(f_1,\ldots,f_r)$ is the ideal of $S$ generated by the
positive-degree homogeneous $W$-invariants.
Chevalley~\cite{Chevalley} proves that $C_W$ affords the regular
representation of $W$. Choosing graded $W$-equivariant lifts of
$C_W$ to $S$ and multiplying by invariants gives an isomorphism
of graded $S^W$-modules and $W$-modules
\begin{equation}\label{b:eq:coinvariant-decomposition}
 S\cong S^W\otimes C_W,
\end{equation}
with $W$ acting trivially on the first factor.
In particular,
\[
 \Hilb(S^W;q)=\prod_{i=1}^r\frac1{1-q^{d_i}},\qquad
 \Hilb(C_W;q)=\prod_{i=1}^r[d_i]_q,\qquad
 \dim C_W=\prod_{i=1}^r d_i=|W|.
\]

For a $W$-module $M$, its \defn{fake degrees}
$e_1(M),\ldots,e_{\dim M}(M)$ record its graded multiplicities in $C_W$:
\[
 \sum_{j\ge0}\dim\Hom_W(M,(C_W)_j)q^j
       =\sum_{i=1}^{\dim M}q^{e_i(M)}.
\]
The regular representation ensures that this list has $\dim M$ entries,
counted with multiplicity. In particular, $e_i(V)=d_i-1$.

For an irreducible module $U$, the following formula tells us in which
degrees to look for hsops carrying $U$, by determining when $U$ occurs
in $S_p$. Such a copy must still have only the origin as its common zero.

\begin{proposition}\label{prop:hsop-multiplicity}
For a $W$-module $U$, put $m_U(p)\defeq\dim\Hom_W(U,S_p)$. Then
\[
 \sum_{p\ge0}m_U(p)q^p
 =\frac{\sum_{i=1}^{\dim U}q^{e_i(U)}}
        {\prod_{i=1}^r(1-q^{d_i})}.
\]
\end{proposition}
\begin{proof}
Apply $\Hom_W(U,-)$ to~\cref{b:eq:coinvariant-decomposition}.
The fake degrees give the multiplicity series in $C_W$, and the
invariant factor has Hilbert series $\prod_i(1-q^{d_i})^{-1}$.
\end{proof}

\begin{example}\label{ex:h3-multiplicity}
Continue \cref{ex:intro-h3} from the introduction, and let
$(V^*)^\sigma$ be the nontrivial Galois twist of $V^*$. The invariant degrees are
$2,6,10$, whereas the two fake-degree lists are $1,5,9$ and $3,5,7$.
Thus
\begin{align*}
 \sum_{p\ge0}m_{V^*}(p)q^p
 &=\frac{q+q^5+q^9}{(1-q^2)(1-q^6)(1-q^{10})}
   =q+q^3+2q^5+3q^7+4q^9+6q^{11}+O(q^{13}),\\
 \sum_{p\ge0}m_{(V^*)^\sigma}(p)q^p
 &=\frac{q^3+q^5+q^7}{(1-q^2)(1-q^6)(1-q^{10})}
   =q^3+2q^5+3q^7+4q^9+5q^{11}+O(q^{13}).
\end{align*}
Both modules occur with multiplicity two in degree five; we construct
hsops of both types in~\cref{c:lem:seeds}. The common-zero condition
still needs checking: the unique copy of $V^*$ in degree three is
$QV^*$, where $Q$ is a quadratic invariant, and it vanishes on the
nonzero points of $Q=0$. On the other hand, the unique twisted copy in degree three is
an hsop by~\cref{c:thm:GG}, since $3$ is coprime to $10$ and
$\sigma_3$ restricts to the nontrivial automorphism of $K_{H_3}$.
\end{example}

\subsection{The quotient character}\label{b:sec:quotient-character}

The equivariant Koszul calculation gives the graded character
of the quotient $A_\Theta=S/(\Theta)$ from~\cref{def:parking_space}; see
Bessis--Reiner~\cite[Proposition~4.2 and its proof]{BR}.
We recall the argument for a $W$-stable hsop $\Theta$.
This formula gives
$(1)\Rightarrow(2)$ in \cref{c:thm:main}; we also use it to calculate
invariant multiplicities in \cref{a:thm:automatic} and the ungraded
parking character in \cref{c:thm:untwisted}.

Let $\Theta\subseteq S_p$ be a $W$-stable hsop. Evaluation defines a
homogeneous equivariant map
\[
 V\longrightarrow\Theta^*,\qquad
 v\longmapsto\bigl(\theta\longmapsto\theta(v)\bigr).
\]
Its zero fibre is the common zero set of $\Theta$, which is the
origin. Equivalently, $A_\Theta$ is finite-dimensional. The evaluation
map is finite: homogeneous lifts of a basis of $A_\Theta$ generate
$S$ over the subalgebra $\C[\Theta]$ by induction on degree.
Conversely, a finite evaluation map has a finite zero fibre; homogeneity
forces this fibre to be the origin.
A basis $\theta_1,\ldots,\theta_r$ of $\Theta$ is a regular sequence,
meaning that $\theta_i$ is a nonzerodivisor modulo its predecessors.
Indeed, $S$ is Cohen--Macaulay and these $r$ elements form a system
of parameters~\cite[Lemmas~10.104.2 and~10.104.7]{Stacks}.

\begin{proposition}\label{prop:koszul}
Let $\Theta\subseteq S_p$ be a $W$-stable hsop. Then
\[
 \Hilb(A_\Theta;q)=[p]_q^r,
 \qquad
 \chi_{A_\Theta}(w;q)=
 \frac{\det(1-q^p w\mid\Theta)}{\det(1-qw\mid V^*)}.
\]
\end{proposition}
\begin{proof}
The Koszul complex has terms
$K_j\defeq S\otimes\bigwedge^j\Theta$, with each element of $\Theta$
in its original degree $p$, and differential
\[
 d_j(f\otimes\theta_1\wedge\cdots\wedge\theta_j)
 =\sum_{a=1}^j(-1)^{a-1}f\theta_a\otimes
       \theta_1\wedge\cdots\widehat{\theta_a}\cdots\wedge\theta_j.
\]
The differential preserves total degree. The action on $K_j$ is diagonal,
and $w(f\theta_a)=(wf)(w\theta_a)$ shows that $d_j$ commutes with $W$.
Regularity makes this complex exact in positive degrees~\cite[Lemma~15.31.2]{Stacks},
giving the exact sequence of graded $W$-modules
\[
 0\longrightarrow K_r\longrightarrow\cdots\longrightarrow K_1
 \longrightarrow S\longrightarrow A_\Theta\longrightarrow0.
\]
Its alternating graded character is
\[
 \chi_{A_\Theta}(w;q)
 =\chi_S(w;q)\sum_{j=0}^r(-q^p)^j\chi_{\wedge^j\Theta}(w).
\]
Diagonalizing $w$ gives
$\chi_S(w;q)=\det(1-qw\mid V^*)^{-1}$ and identifies the sum with
$\det(1-q^pw\mid\Theta)$. At $w=1$, the quotient is
$(1-q^p)^r/(1-q)^r=[p]_q^r$.
\end{proof}

\subsection{Exterior invariants}\label{b:sec:exterior}

The character of $A_\Theta$ in \cref{prop:koszul} is an alternating
sum of the characters of $S\otimes\bigwedge^j\Theta$. To compute its
determinant multiplicity in \cref{a:thm:automatic},
we use the exterior-algebra theorem of
Orlik--Solomon~\cite[Theorem~3.1]{OS} in the following form
(see also~\cite[Theorem~2.7]{AW}).

Recall that a $W$-module $M$ is \defn{amenable} if
$\sum_i e_i(M)=e_1(\det M)$, where
$\det M\defeq\bigwedge^{\dim M}M$ is its one-dimensional determinant
representation.

\begin{theorem}[{Orlik--Solomon~\cite[Theorem~3.1]{OS}}]
\label{b:thm:OS}
If $M$ is amenable, then $(S\otimes\bigwedge M^*)^W$ is an exterior
algebra over $S^W$ on $\dim M$ generators with bidegrees
$(e_i(M),1)$. Consequently,
\[
 \Hilb\bigl((S\otimes\textstyle\bigwedge M^*)^W;q,u\bigr)
 =\frac{\prod_{i=1}^{\dim M}(1+uq^{e_i(M)})}
        {\prod_{i=1}^r(1-q^{d_i})},
\]
where $q$ records polynomial degree and $u$ records exterior degree.
\end{theorem}

\subsection{Galois twists}\label{b:sec:galois}


As in the introduction, let $K_W\defeq\Q(\chi_V(w):w\in W)$.
Its \defn{conductor} $\mathfrak f_W$ is the least positive integer
such that $K_W\subseteq\Q(\zeta_{\mathfrak f_W})$. Every irreducible
$W$-module can be realized over $K_W$ by Benard--Bessis~\cite{Benard,BessisField};
see also Marin--Michel~\cite[Introduction]{MarinMichel}.
For $\gcd(k,\mathfrak f_W)=1$, let $\sigma_k$ send
$\zeta_{\mathfrak f_W}$ to $\zeta_{\mathfrak f_W}^k$.
Applying it to a matrix realization defines the \defn{Galois twist}
$M^{\sigma_k}$; equivalently,
\[
 \chi_{M^{\sigma_k}}(w)=\sigma_k(\chi_M(w)).
\]
Only the restriction to $K_W$ affects the isomorphism class, so
different $k$ can specify the same twist. The notation $\sigma_p$
requires $\gcd(p,\mathfrak f_W)=1$; otherwise the twist must be
specified independently of the degree. The conductor $\mathfrak f_W$
measures the cyclotomic field needed to define the twists; the degree
period $M$ records when a fixed twist occurs as an hsop. For example,
$K_{H_3}=\Q(\sqrt5)$ has conductor $5$, whereas its hsop degrees have
period $10$ in \cref{c:tab:exceptional}: the degree conditions also
exclude even integers.

When $\Theta$ carries $(V^*)^\sigma$, its evaluation map has
target $\Theta^*\cong V^\sigma$. This accounts for the two fake-degree lists in
\cref{c:thm:main}. For polynomial maps we extend $\sigma$ to an
automorphism of $\C$ before applying it to coefficients.

\begin{theorem}[{Marin--Michel~\cite[Theorem~1.2]{MarinMichel}}]
\label{b:thm:MarinMichel}
Let $W\subseteq\operatorname{GL}(V)$ be an irreducible complex
reflection group. If $M$ is a faithful irreducible $W$-module on which
every reflection of $W$ acts as a reflection, then $M\cong V^\sigma$
for some $\sigma\in\operatorname{Gal}(K_W/\Q)$.
\end{theorem}

The irreducibility hypothesis on the module can be removed by allowing
a fixed summand. We will apply the following consequence to hsop
generators in \cref{a:cor:faithful-hsop}.

\begin{proposition}\label{a:prop:faithful-reflection-module}
Let $W\subseteq\operatorname{GL}(V)$ be a nontrivial irreducible complex
reflection group. If $M$ is a faithful $W$-module on which every
reflection of $W$ acts as a reflection, then
\[
 M\cong M^W\oplus V^\sigma
\]
for some $\sigma\in\operatorname{Gal}(K_W/\Q)$.
\end{proposition}
\begin{proof}
Write $M=M^W\oplus M_1\oplus\cdots\oplus M_k$ with the $M_i$
nontrivial and irreducible. Each reflection acts nontrivially on exactly
one $M_i$. Let $W_i$ be generated by the reflections assigned to $M_i$.
Each $W_i$ acts trivially on $M^W$ and on all $M_j$ with $j\ne i$.
Thus distinct factors commute on $M$, and any element of
$W_i\cap\prod_{j\ne i}W_j$ acts trivially on every summand of $M$.
Faithfulness implies that the $W_i$ commute in $W$ and that these
intersections are trivial. Since the reflections generate $W$, we have
$W=W_1\times\cdots\times W_k$.
Write the irreducible reflection representation as
$V=V_1\boxtimes\cdots\boxtimes V_k$. For a generating reflection
$s\in W_i$,
\[
 1=\operatorname{rank}(s-1\mid V_i)\prod_{j\ne i}\dim V_j.
\]
If $k>1$, these identities force every $V_i$ to be one-dimensional.
But then every nonidentity element of $W$ is a reflection, whereas a
product of nonidentity elements from two factors acts nontrivially on
two summands of $M$. Therefore $k=1$, and \cref{b:thm:MarinMichel}
applies to $M_1$.
\end{proof}

\subsection{Untwisted parking characters}\label{b:sec:ito-okada}

Ito--Okada~\cite[Theorem~1.4]{IO} classify the degrees for which the
untwisted determinant quotient is a graded character, and relate them
to the permutation character $w\mapsto p^{\dim V^w}$.
Their theorem, recalled below, supplies the untwisted implication in
\cref{c:thm:untwisted}. Our proof of \cref{c:thm:main} realizes these
graded characters as quotients by posh hsops of degree $p$.

\begin{theorem}[Ito--Okada~{\cite[Theorem 1.4]{IO}}]\label{b:thm:IO}
Let $W$ be an irreducible finite complex reflection group and $p$ a
positive integer. The class function
\[
 w\longmapsto
 \frac{\det(1-q^p w\mid V^*)}{\det(1-qw\mid V^*)}
\]
is the character of a finite-dimensional graded $W$-module if and only
if both $\Cat^+_{p,V^*}(q)$ and $\Cat^-_{p,V^*}(q)$ from
\cref{c:thm:main} are polynomials. Whenever these conditions hold,
$w\mapsto p^{\dim V^w}$ is a permutation character. For non-dihedral
$W$, the converse also holds.
\end{theorem}
\begin{proof}
Ito--Okada~\cite[Theorem 1.4(A)]{IO} use $V$ in their determinant
quotient. Replacing $w$ by $w^{-1}$ gives our quotient and dualizes
each homogeneous representation. Specializing at $q=1$ gives
$w\mapsto p^{\dim V^w}$. Parts (B) and (C) of their theorem imply
that whenever this is a character, it is a permutation character.
Part (B) gives the converse outside dihedral type.
\end{proof}

\subsection{Earlier hsop constructions}\label{c:sec:cherednik}

We review the Gordon--Griffeth construction supplying the well-generated
coprime cases, and the composition that extends exceptional
seed maps to further degrees. As an application of this composition,
we review Haiman's construction in type $A$ in \cref{c:ex:haiman}.

\begin{theorem}[{Gordon--Griffeth~\cite[Theorem~2.11 and Section~2.18]{GG}}]
\label{c:thm:GG}
Let $W$ be irreducible and well-generated, with Coxeter number $h=d_r$,
and let $p>0$ be relatively prime to $h$. Then
$(p,(V^*)^{\sigma_p})$ is posh, where $\sigma_p(\zeta_h)=\zeta_h^p$.
\end{theorem}

\begin{proof}[Construction]
BMR freeness in characteristic zero makes the generic Hecke algebra
of $W$ free of rank $|W|$ over its parameter ring, supplying the
freeness hypothesis in Gordon--Griffeth's construction. See
Etingof~\cite[Theorem~1.3]{Etingof} for the deduction from results of
Marin~\cite{MarinCubic,MarinFreeness} and
Marin--Pfeiffer~\cite[Theorem~1.2]{MarinPfeiffer}.

In the sign convention of Gordon--Griffeth~\cite[Sections~2.9--2.11]{GG},
the polynomial representation of the rational Cherednik algebra at
parameter $-1/h$ has the trivial module as a quotient, resolved by
its coordinate functions. Their equivalence of categories transports
this resolution to parameter $-p/h$. Writing $\Delta(M)$ for the
standard module, whose underlying $S$-module is $S\otimes M$, the last
part is
\[
 \Delta((V^*)^{\sigma_p})\longrightarrow
 \Delta(\mathbf1)=S\longrightarrow L\longrightarrow0.
\]
The singular generating copy of $(V^*)^{\sigma_p}$ maps into $S_p$
and generates the kernel. Since $L$ is finite-dimensional, its image
is an hsop. Thus the construction supplies a parameter space inside
$S$ without specifying polynomial coordinates.
\end{proof}

\begin{lemma}\label{c:lem:composition}
If $(a,(V^*)^\sigma)$ and $(b,(V^*)^\tau)$ are posh, then so is
$(ab,(V^*)^{\sigma\tau})$.
\end{lemma}
\begin{proof}
Choose corresponding hsop maps $F:V\to V^\sigma$ and $G:V\to V^\tau$.
Applying $\sigma$ to the coefficients and covariance equations of $G$
gives $G^\sigma:V^\sigma\to V^{\sigma\tau}$. The composite
$G^\sigma\circ F$ is equivariant and homogeneous of degree $ab$.
Its zero fibre is $F^{-1}((G^\sigma)^{-1}(0))=F^{-1}(0)=\{0\}$.
Galois actions on the character field are abelian, so its representation
type depends only on the two twist classes.
\end{proof}

\begin{definition}\label{b:def:seed}
A \defn{seed hsop} is a posh hsop used with~\cref{c:lem:composition}
to construct posh hsops in further degrees. We will give exceptional seed
hsops in \cref{c:sec:seeds} and use them in \cref{c:sec:noncoprime}
to obtain the remaining posh pairs for
$G_4,G_6,H_3,G_{27}$.
\end{definition}

\begin{example}\label{c:ex:haiman}
Let $W=S_n$ and $V=\{x\in\C^n:\sum_i x_i=0\}$.
For a prime $\ell\nmid n$, Haiman~\cite[Proposition~2.5.4]{Haiman}
uses the projected power map
\begin{equation}\label{c:eq:haiman-projection}
 H_\ell(x)_i\defeq x_i^\ell-\frac1n\sum_jx_j^\ell.
\end{equation}
It is homogeneous, $S_n$-equivariant, and takes values in $V$.
If $H_\ell(x)=0$, all $x_i^\ell$ agree. A nonzero zero can therefore
be rescaled so that every $x_i$ is an $\ell$th root of unity.
Writing $x_i=\zeta_\ell^{a_i}$ gives
$\sum_i\zeta_\ell^{a_i}=0$. The polynomial
$\sum_i t^{a_i}$ is divisible by $1+t+\cdots+t^{\ell-1}$;
evaluation at $t=1$ forces $\ell\mid n$, a contradiction.
Factoring $p$ into primes and composing these maps proves that
$(p,V^*)$ is posh whenever $\gcd(p,n)=1$.
\end{example}

\section{Generator representations and parking characters}\label{sec:generators}

For irreducible $W$, we identify the generator modules of posh hsops
as Galois twists of $V^*$ in~\cref{a:cor:faithful-hsop} and compute
both Catalan multiplicities in~\cref{a:thm:automatic}.
We then determine which parking spaces are permutation modules
in~\cref{c:thm:untwisted}. These restrictions are used to classify
posh pairs in~\cref{sec:proof-algebraic}.

\subsection{Generator restriction and multiplicities}\label{sec:reflection-restrictions}

We first give a local criterion for applying the exterior-invariant
theorem of~\cref{b:thm:OS}. The hsop condition will force this criterion to
hold, allowing us to compute both multiplicities and identify the
faithful generator modules.

For a reflecting hyperplane $H$, let
$W_H\defeq\{w\in W:w|_H=1\}$ be its cyclic pointwise stabilizer,
and put $e_H\defeq|W_H|$. Choose its generator $s_H$ so
that its eigenvalue on $V/H$ is $\zeta_H\defeq\exp(2\pi i/e_H)$.
The local exponent formula of E.~A.~Gutkin and E.~M.~Opdam~\cite[Lemma 2.1]{Opdam}
and V.~Beck's criterion~\cite[Definition 4(ii) and Remark 1(i)]{Beck}
give the following characterization.

\begin{proposition}[{\cite[Lemma 2.1]{Opdam}, \cite[Remark 1(i)]{Beck}}]
\label{a:prop:two-sided-amenability}
For a finite-dimensional $W$-module $M$, the following are equivalent:
\begin{enumerate}
\item both $M$ and $M^*$ are amenable;
\item every reflection $s$ satisfies $\operatorname{rank}(1-s\mid M)\leq1$;
\item every reflection acts on $M$ as a reflection or the identity.
\end{enumerate}
\end{proposition}
\begin{proof}
Write
\[
 m_{H,j}\defeq\dim\ker(s_H-\zeta_H^{-j}I\mid M),
 \qquad 0\leq j<e_H,
\]
so $m_{H,j}$ is the multiplicity of the indicated eigenvalue, and put
\[
 J_H(M)\defeq\sum_{j=1}^{e_H-1}j m_{H,j},\qquad
 a_H(M)\defeq\sum_{j=1}^{e_H-1}m_{H,j}.
\]
Let $\overline J_H(M)$ be the least nonnegative residue of $J_H(M)$
modulo $e_H$. The sum formula of Gutkin--Opdam~\cite[Lemma 2.1]{Opdam},
\[
 \sum_i e_i(M)=\sum_H J_H(M),
\]
extends from irreducible modules to arbitrary $M$ by additivity.
Applying its one-dimensional case to $\det M$ gives
$e_1(\det M)=\sum_H\overline J_H(M)$, and therefore
\[
 \sum_i e_i(M)-e_1(\det M)
 =\sum_H\bigl(J_H(M)-\overline J_H(M)\bigr).
\]
Every summand on the right is nonnegative, so $M$ is
amenable exactly when $J_H(M)<e_H$ for every $H$. Also
\[
 J_H(M^*)=e_Ha_H(M)-J_H(M).
\]
Consequently both modules are amenable exactly when $a_H(M)\leq1$
for every $H$: if both local sums are less than $e_H$, their sum
$e_Ha_H(M)$ is less than $2e_H$; conversely, at most one nontrivial
eigenvalue makes both sums less than $e_H$. This is Beck's local
criterion~\cite[Remark 1(i)]{Beck}, and is equivalent to (2) and (3).
\end{proof}

Multiplication sends basic covariants to generators of
$(\Theta)\cap S^W$, explaining the numerator degrees in the invariant
Hilbert series. These invariant parameters also provide examples
in~\cref{sec:questions}.

\begin{proposition}\label{b:prop:invariant-parameters}
Let $\Theta\subseteq S_p$ be a $W$-stable hsop. The ideal
$(\Theta)\cap S^W$ is generated by a regular sequence of degrees
$p+e_i(\Theta^*)$, so
\begin{equation}\label{b:eq:invariant-parameter-hilbert}
 \Hilb(A_\Theta^W;q)
 =\prod_{i=1}^r\frac{[p+e_i(\Theta^*)]_q}{[d_i]_q}.
\end{equation}
\end{proposition}
\begin{proof}
Multiplication gives the $W$-equivariant presentation
\[
 S\otimes\Theta\xrightarrow{\ \mu\ }S
 \longrightarrow A_\Theta\longrightarrow0,
 \qquad \mu(f\otimes\theta)=f\theta.
\]
Taking $W$-invariants is exact, so
\[
 (S\otimes\Theta)^W\xrightarrow{\ \mu\ }S^W
 \longrightarrow A_\Theta^W\longrightarrow0.
\]
By~\cref{b:eq:coinvariant-decomposition}, $(S\otimes\Theta)^W$
is a free $S^W$-module of rank $r$. Its homogeneous basis
$c_1,\ldots,c_r$ has total degrees $p+e_i(\Theta^*)$: the polynomial
coefficients have degrees $e_i(\Theta^*)$, and $\Theta$ has degree $p$.
Such invariant tensors are called \defn{covariants}; equivalently,
their polynomial coefficients define equivariant maps $V\to\Theta$.
Set
\[
 g_i\defeq\mu(c_i),\qquad \deg g_i=p+e_i(\Theta^*).
\]
The image of $\mu$ is $(\Theta)\cap S^W$, so
\[
 A_\Theta^W\cong S^W/(g_1,\ldots,g_r).
\]
This quotient is finite-dimensional. Thus $g_1,\ldots,g_r$ form a
system of parameters in the polynomial ring $S^W$ and therefore a
regular sequence~\cite[Lemmas~10.104.2 and~10.104.7]{Stacks}.
Their degrees give the stated Hilbert series.
\end{proof}

\begin{example}\label{ex:invariant-parameters}
For an untwisted hsop, choose a $W$-isomorphism $V^*\cong\Theta$
taking a coordinate basis $x_1,\ldots,x_r$ to
$\theta_1,\ldots,\theta_r$. The theorem of Solomon~\cite[Theorem~2.3]{AW}
on invariant differential forms gives the covariant basis
$df_1,\ldots,df_r$ for homogeneous basic invariants $f_i$.
The invariant parameters are
\[
 g_i=\sum_{j=1}^r\frac{\partial f_i}{\partial x_j}\theta_j,
 \qquad \deg g_i=p+d_i-1.
\]
For well-generated $W$ and $p=mh+1$ with $m\ge0$ an integer,
the Hilbert series becomes the
Fuss $q$-Catalan product $\prod_i[mh+d_i]_q/[d_i]_q$.
\end{example}

If $\Theta$ carries $(V^*)^\sigma$, then
$e_i(\Theta^*)=e_i(V^\sigma)$. In this case the Hilbert series is
the formula of Bessis--Reiner~\cite[Proposition~4.2]{BR}.

We now show that the hsop condition forces the local reflection
criterion in~\cref{a:prop:two-sided-amenability}, and compute the
determinant multiplicity alongside the invariant Hilbert series.

\begin{theorem}\label{a:thm:automatic}\label{c:prop:hilbert}\label{cor:products}
Let $W\subseteq\mathrm{GL}(V)$ be a finite complex reflection group of
rank $r$, and let $\Theta\subseteq S_p$, $p>0$, be a $W$-stable hsop.
Every reflection acts on $\Theta$ as a reflection or the
identity, and both $\Theta$ and $\Theta^*$ are amenable. Moreover,
\begin{align}
 \Hilb(A_\Theta^W;q)
 &=\prod_{i=1}^r\frac{1-q^{p+e_i(\Theta^*)}}{1-q^{d_i}},\label{a:eq:product}\\
 \Hilb\bigl(\Hom_W(\det\Theta,A_\Theta);q\bigr)
 &=\prod_{i=1}^r\frac{q^{e_i(\Theta)}-q^p}{1-q^{d_i}}.\label{c:eq:det-hilbert}
\end{align}
Both products are polynomials with nonnegative integer coefficients. At
$q=1$, the first gives
$\dim A_\Theta^W=|W|^{-1}\prod_{i=1}^r(p+e_i(\Theta^*))$.
\end{theorem}
\begin{proof}
Let $s$ be a reflection with reflecting hyperplane $H$, and choose
an eigenbasis of $\Theta$ for $s$. If $s\theta=\lambda\theta$
with $\lambda\ne1$, then for $v\in H$ we have
$\lambda\theta(v)=(s\theta)(v)=\theta(s^{-1}v)=\theta(v)$,
so $\theta|_H=0$. Thus the restrictions of the $\dim\Theta^s$ fixed
basis elements already have the origin as
their common zero, because $\Theta$ is an hsop. The radical of the
ideal they generate in $\C[H]$ is therefore the maximal ideal of the
origin, of height $r-1$. Krull's height
theorem~\cite[Lemma~10.60.12(1)]{Stacks} gives
$r-1\leq\dim\Theta^s$.
Since $\dim\Theta=r$, every reflection acts on $\Theta$ as a
reflection or the identity. By~\cref{a:prop:two-sided-amenability},
both $\Theta$ and $\Theta^*$ are amenable.

The invariant product of~\cref{a:eq:product} follows directly from
\cref{b:prop:invariant-parameters}.
For the determinant multiplicity, we follow
Bessis--Reiner~\cite[Proposition~4.2]{BR}, using the exterior-algebra
theorem of Orlik--Solomon~\cite[Theorem~3.1]{OS}, recalled in
\cref{b:thm:OS}.
The natural pairing of complementary exterior powers gives
$\bigwedge^j\Theta\otimes(\det\Theta)^*
\cong\bigwedge^{r-j}\Theta^*$.
Tensoring the Koszul resolution in~\cref{prop:koszul} with the ungraded
representation $(\det\Theta)^*$ and taking $W$-invariants computes
$\Hom_W(\det\Theta,A_\Theta)$. Both operations preserve exactness.
In the following calculation the exterior powers are regarded as
ungraded representations; the factors $q^{pj}$ account for their
degrees in the resolution. Applying~\cref{b:thm:OS} to $\Theta$ gives
\begin{align*}
 \Hilb\bigl(\Hom_W(\det\Theta,A_\Theta);q\bigr)
 &=\sum_{j=0}^r(-q^p)^j
   \Hilb\bigl((S\otimes\textstyle\bigwedge^{r-j}\Theta^*)^W;q\bigr)\\
 &=\prod_{i=1}^r\frac{q^{e_i(\Theta)}-q^p}{1-q^{d_i}}.
\end{align*}
Indeed, the term indexed by $j$ selects $r-j$ exterior generators.
Each selected generator contributes $q^{e_i(\Theta)}$, and each of
the $j$ omitted generators contributes $-q^p$. Summing over all
subsets gives the product.
These are Hilbert series of graded multiplicity spaces, so their
coefficients are nonnegative integers. Taking $q\to1$ in
\cref{a:eq:product} and using $\prod_i d_i=|W|$ proves the dimension
formula.
\end{proof}

When $\Theta$ carries $U=(V^*)^\sigma$, the two products are
$\Cat^+_{p,U}$ and $\Cat^-_{p,U}$ from~\cref{c:thm:main}.

\faithfulhsoplemma*

\begin{proof}
A Galois twist of $V^*$ is faithful. Conversely, suppose that $U$ is
faithful. By \cref{a:thm:automatic}, every reflection acts on $U^*$ as a
reflection. Thus \cref{a:prop:faithful-reflection-module} gives
$U^*\cong (U^*)^W\oplus V^\sigma$. Since $\dim U=r$, the fixed
summand vanishes, and dualizing gives $U\cong(V^*)^\sigma$.
\end{proof}

The same conclusion extends to faithful homogeneous regular sequences.
\begin{corollary}\label{a:cor:faithful-partial-sequence}
Let $W\subseteq\operatorname{GL}(V)$ be a nontrivial irreducible complex
reflection group of rank $r$. If $\Theta\subseteq S_p$, $p>0$, is a
faithful $W$-stable subspace whose basis is a regular sequence, then
$\dim\Theta=r$ and $\Theta$ is a posh hsop carrying a Galois
twist of $V^*$.
\end{corollary}
\begin{proof}
Put $\ell\defeq\dim\Theta$. Regularity gives $\ell\le r$ and
$\dim Z=r-\ell$ for $Z\defeq V(\Theta)$. For a reflection $s$
with hyperplane $H$, choose an
$s$-eigenbasis of $\Theta$. Every nonfixed eigenvector vanishes on $H$, so
$Z\cap H$ is cut out inside $H$ by the restrictions of the fixed
eigenvectors. Therefore
\[
 r-1-\dim\Theta^s\leq\dim(Z\cap H)\leq r-\ell,
\]
and $\dim\Theta^s\geq\ell-1$. Thus every reflection acts on $\Theta^*$ as a reflection or
the identity, and faithfulness excludes the identity. By
\cref{a:prop:faithful-reflection-module}, $\Theta^*$ contains a summand
isomorphic to $V^\sigma$, so $\ell\geq r$. Therefore $\ell=r$ and $Z$
has dimension zero. Since $Z$ is a cone containing the origin,
$Z=\{0\}$ and $\Theta$ is an hsop. The fixed summand vanishes, and
dualizing proves the claim.
\end{proof}

\begin{remark}\label{rem:faithful-scope}
Thus \cref{c:thm:main} classifies the degrees and generator modules
of all faithful $W$-stable regular sequences of common positive degree:
each has length $r$ and carries a Galois twist.
The nonfaithful parameter problem is discussed in
\cref{sec:beyond-hsops}.
\end{remark}

\begin{corollary}\label{cor:formal-products}
Let $U=(V^*)^\sigma$. If the determinant quotient in
\cref{c:thm:main}(2) is the character of a finite-dimensional graded
$W$-module, then both $\Cat^+_{p,U}$ and $\Cat^-_{p,U}$ belong to
$\Z_{\geq0}[q]$.
\end{corollary}
\begin{proof}
A Galois twist of a reflection is again a reflection, so
\cref{a:prop:two-sided-amenability} proves amenability of $U$ and
$U^*$ independently of hsop existence. Expanding the numerator as
in~\cref{prop:koszul} gives the same alternating exterior-power series
used in the proof of \cref{a:thm:automatic}. This calculation requires
only the determinant quotient, so it applies without an hsop.
Taking its trivial and $\det U$ multiplicities and applying
\cref{b:thm:OS} gives $\Cat^+_{p,U}$ and $\Cat^-_{p,U}$, respectively.
If the quotient is the character of a finite-dimensional graded module,
these multiplicities are polynomials with nonnegative integer
coefficients. This gives the polynomiality part of $(2)\Rightarrow(3)$
in \cref{c:thm:main} once $U$ is known to be a Galois twist.
\end{proof}

\subsection{Untwisted parameters and permutation modules}\label{sec:permutation}

We now prove that a parking space is a permutation module exactly
when its parameters carry $V^*$. 

The \defn{reduced determinant}
$\rdet(1-w)$ of a finite-order operator $w$ is the product of the factors
$1-\lambda$ with $\lambda\ne1$. Let $\Theta$ be a posh hsop of degree $p$
carrying $U=(V^*)^\sigma$, as in \cref{a:cor:faithful-hsop}.
In the character formula of \cref{prop:koszul}, the numerator and
denominator have the same number $\dim V^w$ of vanishing factors as
$q\to1$. Each pair contributes $p$, while Galois conjugation sends
the remaining eigenvalues on $V^*$ to those on $U$. Therefore
\begin{equation}\label{eq:twisted-character}
 \chi_{A_\Theta}(w)=p^{\dim V^w}
 \frac{\sigma\bigl(\rdet(1-w\mid V^*)\bigr)}
      {\rdet(1-w\mid V^*)}.
\end{equation}
This is the hsop character formula in
Bessis--Reiner~\cite[Proposition~4.2, equation~(4.2)]{BR}.
The untwisted specialization $w\mapsto p^{\dim V^w}$ also appears in
the Orlik--Solomon expansion into parabolic permutation characters~\cite{OS82,OS83};
see Ito--Okada~\cite[Proposition~2.3]{IO}.
The exterior-invariant identity of Orlik--Solomon~\cite[Theorem~3.1]{OS}
underlies the multiplicity products in \cref{a:thm:automatic}; here
we obtain the character value directly by specializing
\cref{prop:koszul}.

\permutationtheorem*
\begin{proof}
Suppose that $A_\Theta$ is a permutation module, and write
$U=(V^*)^\sigma$. If $w$ is \defn{elliptic}, meaning $V^w=0$, then
\cref{eq:twisted-character} gives
\[
 \chi_{A_\Theta}(w)=
 \frac{\sigma\bigl(\det(1-w\mid V^*)\bigr)}
      {\det(1-w\mid V^*)}.
\]
This is a nonzero permutation-character value, therefore a positive
integer. Its norm from $K_W$ to $\Q$ is one, so it equals one.
On nonelliptic elements $\det(1-w\mid V^*)=0$. Consequently
\[
 \sigma\bigl(\det(1-w\mid V^*)\bigr)=\det(1-w\mid V^*)
 \qquad(w\in W).
\]

If $r=1$, then $\det(1-w\mid V^*)=1-\chi_{V^*}(w)$, so
$(V^*)^\sigma\cong V^*$. Suppose $r>1$.
The exterior powers $\bigwedge^jV^*$, $0\leq j\leq r$, are irreducible
and pairwise nonisomorphic by H.~Hu~\cite[Theorem~1.2]{Hu}. Therefore
\[
 \det(1-w\mid V^*)
 =\sum_{j=0}^r(-1)^j\chi_{\wedge^jV^*}(w)
\]
is an expansion into distinct irreducible characters. Galois
conjugation preserves irreducibility and distinctness. Since this
virtual character is fixed by $\sigma$, the character of
$(V^*)^\sigma$ occurs in the expansion with coefficient $-1$.
Thus $(V^*)^\sigma\cong\bigwedge^jV^*$ for an odd $j$.
Dimensions give $\binom rj=r$, so $j=1$ or $j=r-1$.
For $r>2$, a reflection with nontrivial eigenvalue $\lambda$ on $V^*$
has eigenvalue $\lambda$ with multiplicity $r-1$ on
$\bigwedge^{r-1}V^*$. Its fixed space there has codimension $r-1$,
whereas its fixed space on $(V^*)^\sigma$ has codimension one.
This excludes $j=r-1$.
For $r=2$, the two possibilities coincide. Therefore $U\cong V^*$.

Conversely, if $U\cong V^*$, then \cref{eq:twisted-character} gives
$\chi_{A_\Theta}(w)=p^{\dim V^w}$. Ito--Okada~\cite[Theorem 1.4]{IO} show
that this is a permutation character whenever the untwisted graded
character exists, including the dihedral case; see \cref{b:thm:IO}.
\end{proof}

\begin{example}\label{ex:h3-two-types}
Continue \cref{ex:h3-multiplicity}, with $\sigma(\sqrt5)=-\sqrt5$. We will construct posh hsops of degree $5$ in \cref{c:lem:seeds}
that give quotients of vector-space dimension $125$ by \cref{prop:koszul}.
By~\cref{a:thm:automatic}, their invariant Hilbert series are
\begin{align*}
 \Cat^+_{5,V^*}(q)
 &=\frac{1-q^{14}}{1-q^2}
   =1+q^2+q^4+q^6+q^8+q^{10}+q^{12},\\
 \Cat^+_{5,(V^*)^\sigma}(q)
 &=\frac{(1-q^8)(1-q^{12})}{(1-q^2)(1-q^6)}
   =1+q^2+q^4+2q^6+q^8+q^{10}+q^{12}.
\end{align*}
Their invariant dimensions are $7$ and $8$, respectively. Both
determinant multiplicities vanish, since both fake-degree lists contain $5$.

View $H_3$ as the full symmetry group of the icosahedron. Its longest
element is the central inversion $w_0=-I$. Let $R$ be a rotation through
$2\pi/5$ about an axis through opposite vertices, and put
$w\defeq w_0R=-R$. Thus $w$ has order ten and $w^5=w_0$; in particular,
$w$ is not the longest element. Its eigenvalues on $V^*$ are
$-1,-\zeta_5,-\zeta_5^{-1}$, so $w$ is elliptic and
\[
 \rdet(1-w\mid V^*)=2(1+\zeta_5)(1+\zeta_5^{-1})=3+\sqrt5.
\]
For the twisted quotient, \cref{eq:twisted-character} gives
\[
 \chi_{A_\Theta}(w)=\frac{3-\sqrt5}{3+\sqrt5}
                  =\frac{7-3\sqrt5}{2}.
\]
This irrational value excludes a permutation module. The untwisted
quotient is a permutation module by \cref{c:thm:untwisted}.
\end{example}

There is also a converse for non-dihedral groups:
Ito--Okada~\cite[Theorem 1.4]{IO} show that $w\mapsto p^{\dim V^w}$ is a character if and
only if it is a permutation character, and these conditions give
exactly the untwisted degrees in \cref{tab:infinite-untwisted,tab:exceptional-untwisted}.
For dihedral groups the character and permutation-character
conditions are still equivalent, but their range can be larger.
For $I_2(m)$, Ito--Okada's permutation range is $p=1$,
or $p\ge m-1$ with
\[
 p^2\equiv1\pmod m\quad(m\text{ odd}),\qquad
 p^2\equiv1\pmod{2m}\quad(m\text{ even}).
\]
This can exceed the range $p\equiv\pm1\pmod m$ for untwisted posh hsops.

\begin{example}\label{c:ex:dihedral-permutation-without-hsop}
For $I_2(15)$ and $p=19$, the class function
$w\mapsto19^{\dim V^w}$ is a permutation character, but there is no
untwisted posh hsop of degree $19$. The twisted
posh hsop of degree $19$ exists, and its quotient has a different, non-permutation character.
\end{example}

The package~\cite{Computations} also records the permutation obstruction
for exceptional conjugacy classes, independently of the proof of
\cref{c:thm:untwisted}.

\section{Classification of posh pairs}\label{sec:proof-algebraic}\label{sec:computations}

\subsection{Reduction and proof strategy}\label{c:sec:reduction}

We reduce to Galois twists, determine the pairs satisfying the
polynomiality conditions, and construct hsops to prove they are posh.

\classificationtheorem*

\begin{proposition}\label{c:prop:reduction}
Under the hypotheses of \cref{c:thm:main}, the implications
$(1)\Rightarrow(2)\Rightarrow(3)$ hold. Moreover, condition (3) forces
$U\cong(V^*)^\sigma$ for some Galois automorphism $\sigma$.
\end{proposition}

\begin{proof}
\Cref{prop:koszul} gives $(1)\Rightarrow(2)$.
Suppose that (2) holds. For a reflection $s$, the denominator of the
determinant quotient vanishes to order $r-1$ at $q=1$, while the
numerator vanishes to order $\dim U^s$. Polynomiality therefore gives
$\dim U^s\geq r-1$. Since $U$ is faithful and has dimension $r$,
every reflection acts on $U$ as a reflection.
By~\cref{a:prop:faithful-reflection-module}, applied to $U^*$,
we obtain $U\cong(V^*)^\sigma$ for some $\sigma$.
Now \cref{cor:formal-products} gives the two polynomiality conditions,
proving $(2)\Rightarrow(3)$.
Under (3), the reflection condition, faithfulness, and
$\dim U=r$ likewise force $U$ to be a Galois twist of $V^*$.
\end{proof}

It remains to prove $(3)\Rightarrow(4)\Rightarrow(1)$.
\Cref{c:sec:imprimitive} proves both implications for the infinite
families. For the exceptional groups,
\cref{c:sec:certificates} determines all possible pairs, and
\cref{tab:exceptional-methods-overview} organizes their constructions.
We assemble these results in \cref{c:sec:completion}.

\begin{table}[htbp]
\centering
\begin{tabular}{@{}L{.61\textwidth}L{.31\textwidth}@{}}
\toprule
Groups or parameters & Construction\\
\midrule
Well-generated groups, $\gcd(p,h)=1$ & \cref{c:thm:GG}\\
$G_4,G_6,H_3,G_{27}$, remaining degrees & \cref{c:sec:noncoprime}\\
$G_7,G_{11},G_{12},G_{13},$\newline $G_{15},G_{19},G_{22}$ & \cref{c:lem:scalar}\\
$G_{31}$ & \cref{c:lem:eigenspace}\\
\bottomrule
\end{tabular}
\caption{Constructions for the exceptional groups in
\cref{c:thm:main,c:thm:untwisted}, with $h=d_r$ for a
well-generated group. The degree and generator types are listed in
\cref{c:tab:exceptional,tab:exceptional-untwisted}.}
\label{tab:exceptional-methods-overview}
\end{table}

\subsection{Infinite families}\label{c:sec:imprimitive}

For $d\mid m$, let $G(m,d,r)$ be the group of $r\times r$ monomial
matrices whose nonzero entries are $m$th roots of unity and whose
product is an $(m/d)$th root of unity. In rank one we take $d=1$,
so $G(m,1,1)=C_m$. The symmetric group $S_n$ acts on its
rank-$(n-1)$ reflection representation.

\begin{proposition}\label{c:prop:infinite-classification}
Under the hypotheses of \cref{c:thm:main}, condition (3) holds
for the infinite families
exactly for the following pairs:
\begin{enumerate}
\item $W=S_n$: $\gcd(p,n)=1$ and $U\cong V^*$;
\item $W=G(m,d,r)$ irreducible, with $m\ge2$:
$\gcd(p,m)=1$ and $U\cong(V^*)^{\sigma_p}$.
\end{enumerate}
Every listed pair is posh; the list agrees with \cref{tab:infinite-algebraic}.
\end{proposition}

We first construct the coordinate-power hsops and use them to compute
the fake degrees needed for necessity.

\begin{proposition}\label{a:prop:monomial}
Let $W=G(m,d,r)$, where $m\ge2$, $d\mid m$, and $r\ge2$. For
$1\le k\le m$, let $M_k$ be the monomial representation obtained by
replacing each nonzero matrix entry $\lambda$ by $\lambda^k$, and put
$U_k\defeq M_k^*$. For every $p=mj+k$, $j\ge0$, the subspace
\[
 \Theta_p\defeq\operatorname{span}\{x_1^p,\ldots,x_r^p\}
\]
is an hsop carrying $U_k$. Put $a\defeq m/d$ and $c_k\defeq\lceil k/a\rceil$.
Then
\begin{equation}\label{a:eq:fake}
 \{e_i(U_k^*)\}
 =\{m-k,2m-k,\ldots,(r-1)m-k,\;ra c_k-k\}.
\end{equation}
In particular, no coprimality assumption on $k$ is needed.
\end{proposition}

\begin{proof}
Entrywise powers preserve products of monomial matrices, so $M_k$ is
a representation. The coordinate-power map has the stated covariance
and only the origin as its zero. A monomial is fixed by the diagonal
subgroup exactly when all its exponents are congruent to $ba$ modulo
$m$ for some $0\le b<d$. In the range $0,\ldots,mj+k-1$, exactly
$c_k$ of these residue classes have $j+1$ representatives and the others
have $j$. Coordinate permutations therefore give
\[
 \dim A_{\Theta_p}^W=c_k\binom{j+r}{r}+(d-c_k)\binom{j+r-1}{r}
 =\frac{(dj+rc_k)(j+1)\cdots(j+r-1)}{r!}.
\]
Substituting $p=mj+k$ in the dimension formula of~\cref{a:thm:automatic} and using
$|W|=m^r r!/d$ proves~\cref{a:eq:fake} by equality of monic polynomials
for infinitely many $j$.
\end{proof}

\begin{proof}[Proof of \cref{c:prop:infinite-classification}]
By \cref{c:prop:reduction}, it suffices to consider Galois twists of
$V^*$.

\emph{Coordinate-power constructions.}
For $C_m$ and $G(m,d,r)$, write $U=(V^*)^{\sigma_k}$, where
$1\le k<m$, $\gcd(k,m)=1$, and
$\sigma_k(\zeta_m)=\zeta_m^k$. Coordinate powers
\[
 F_p(x_1,\ldots,x_r)\defeq(x_1^p,\ldots,x_r^p)
\]
give the required hsop when $p\equiv k\pmod m$, by
\cref{a:prop:monomial} and its immediate rank-one analogue.
In the dihedral case, swapped coordinate powers give the residue
$-k$ as well.

\emph{Fake degrees.}
The exponents $b_i=e_i(U^*)$ are given by~\cref{a:eq:fake}, with
$a=m/d$ and $c_k=\lceil k/a\rceil$ as in~\cref{a:prop:monomial}.
The $a_i=e_i(U)$ are obtained by replacing $k$ with $m-k$ in that formula.

\emph{Necessary congruences.}
If $r\ge3$, the denominator of $\Cat^+_{p,U}$ has a zero of order at least
$r-1$ at a primitive $m$th root of unity. If $p\not\equiv k\pmod m$,
at most its last numerator factor vanishes, which is impossible.
If $r=2$ and $d<m$, assume again $p\not\equiv k\pmod m$ and set
$c'\defeq c_{m-k}$. Polynomiality of $\Cat^+$ forces
$p-k\equiv-2ac_k\pmod m$, whereas polynomiality of $\Cat^-$ forces
$p-k\equiv2ac'\pmod m$. The latter congruence also holds if $\Cat^-=0$,
since its first exponent is $k$. As $c_k+c'=d+1$, these congruences
force $d\mid2$. But then both invariant degrees $m,2m/d$ are divisible
by $m$, and the numerator of $\Cat^+$ has only one vanishing factor, again
impossible. For $I_2(m)$, the exponents are $k,m-k$, so polynomiality at
a primitive $m$th root forces $p\equiv\pm k$.
The twists indexed by $k$ and $-k$ are isomorphic, so either congruence
gives $U\cong(V^*)^{\sigma_p}$.
The rank-one assertion is immediate.

\emph{Type $A$.}
For $S_n$ there is only the trivial twist. The classification recalled
in~\cref{b:thm:IO} makes (3) equivalent to $\gcd(p,n)=1$.
Haiman's construction
in \cref{c:ex:haiman}, or \cref{c:thm:GG}, supplies the hsop.
\end{proof}

In particular, for the Weyl groups of types $B_r$ and $C_r$, both
realized as $G(2,1,r)$, and of type $D_r=G(2,2,r)$, a posh hsop
of degree $p$ exists exactly when $p$ is odd, and it is untwisted.
For every positive odd $p$, the coordinate powers
$\Theta_p=\operatorname{span}\{x_1^p,\ldots,x_r^p\}$ give such an hsop.
Thus no coprimality condition with $h$ is required in these types:
for example, $p=3$ is allowed in both $B_3$ and $D_4$, whose Coxeter
number is $h=6$.

\subsection{Exceptional polynomiality conditions}\label{c:sec:certificates}

\begin{proposition}\label{c:prop:exceptional-polynomiality}
Let $W$ be an exceptional irreducible complex reflection group and
$U$ a Galois twist of $V^*$. Both products in \cref{c:thm:main}(3)
are polynomials exactly for the pairs $(p,U)$ in \cref{c:tab:exceptional}.
\end{proposition}

\begin{proof}
\emph{Finite reduction.}
For each group and twist put $L\defeq\operatorname{lcm}(d_1,\ldots,d_r)$ and
$A\defeq\max_i a_i$. Above $A$, polynomiality of
$\Cat^+_{p,U}(q)$ and $\Cat^-_{p,U}(q)$
depends only on $p\pmod L$: each factor $1-q^n$ contains the
cyclotomic polynomial $\Phi_\delta(q)$ exactly once when $\delta\mid n$.
We test one representative greater than
$A$ in every residue class, using the inequalities
\[
 \#\{i:\delta\mid p+b_i\}\geq\#\{i:\delta\mid d_i\},
 \qquad
 \#\{i:\delta\mid p-a_i\}\geq\#\{i:\delta\mid d_i\}
\]
for every $\delta$ dividing an invariant degree. We also test every
$1\leq p\leq A$. If $p=a_i$ for some $i$, the second product is zero;
otherwise its cyclotomic multiplicities are obtained from $|p-a_i|$.
Thus finitely many cyclotomic divisibility tests determine all
positive degrees.

\emph{Evaluation.}
The finite calculation is implemented in the
\href{https://personal.utdallas.edu/~nxw170830/docs/algebraic-hsops/}{accompanying computation package}~\cite{Computations},
using the CHEVIE package~\cite{GHLMP,Michel} for reflection-group data.
It checks the $128$ distinct Galois twists of the $34$ exceptional
groups and gives precisely the rows of \cref{c:tab:exceptional},
with no additional small-degree exceptions. Testing the divisors of
$L$ gives the least period of each row.
\end{proof}

\subsection{The four remaining well-generated groups}\label{c:sec:seeds}

Among well-generated exceptional groups, the coprime construction
in~\cref{c:thm:GG} leaves precisely the degrees
$p\equiv3\pmod6$ for $G_4$,
$p\equiv3,9\pmod{12}$ for $G_6$,
$p\equiv5\pmod{10}$ for $H_3$, and
$p\equiv5,25\pmod{30}$ for $G_{27}$, with the generator twists
listed in~\cref{c:tab:exceptional}.
We state eight seed hsops, show how they cover every remaining pair
by composition, and then verify the seeds.

The need for constructions beyond the coprime Cherednik theorem
already appears in $H_3$ at $p=5$. M.~Balagovi\'c and A.~Puranik~\cite[Theorem 10.1]{BalagovicPuranik} compute
that the spherical simple Cherednik module $L_{1/2}(\mathbf1)$ has
dimension $115$; a degree-five
hsop quotient has dimension $125$. Thus the quintic parking modules
cannot be identified with that simple module.

\subsubsection{Seeds and coverage}\label{c:sec:noncoprime}

We use seed hsops in the sense of \cref{b:def:seed}.

\begin{lemma}\label{c:lem:seeds}
The eight pairs $(p,(V^*)^{\sigma_k})$ in \cref{tab:seeds} are posh.
\end{lemma}

\begin{table}[htbp]
\centering
\begin{tabular}{@{}rrrr@{}}
\toprule
$W$ & conductor $\mathfrak f_W$ & degree $p$ & twist representatives $k$\\
\midrule
$G_4$ & $3$ & $3$ & $1,2$\\
$G_6$ & $12$ & $3$ & $7,11$\\
$H_3=G_{23}$ & $5$ & $5$ & $1,2$\\
$G_{27}$ & $15$ & $5$ & $2,11$\\
\bottomrule
\end{tabular}
\caption{The eight seed hsops, carrying $(V^*)^{\sigma_k}$.
Each row contains one map for each listed twist.}\label{tab:seeds}
\end{table}

\begin{proposition}\label{c:prop:seed-coverage}
For $W=G_4,G_6,H_3,G_{27}$, every pair in
\cref{c:tab:exceptional} is realized by composing an hsop from
\cref{c:thm:GG} with zero or more seeds from \cref{c:lem:seeds}.
\end{proposition}

\begin{proof}[Proof assuming \cref{c:lem:seeds}]
The simultaneous polynomiality calculation in \cref{c:prop:exceptional-polynomiality}
gives the following residue lists. Here $k$
specifies the generator twist with conductor $\mathfrak f_W$ as in
\cref{tab:seeds}, and the degree period is the Coxeter number
$h=6,12,10,30$ for $G_4,G_6,H_3,G_{27}$, respectively.
Each entry lists all allowed residues of the positive degree $p$:
\begin{equation}\label{c:eq:residual-table}
\begin{array}{c|c|l}
 W&k& p\text{ modulo }h\\\hline
 G_4&1&1,3\\
    &2&3,5\\\hline
 G_6&1&1,9\\
    &5&5,9\\
    &7&3,7\\
    &11&3,11\\\hline
 H_3&1&1,5,9\\
    &2&3,5,7\\\hline
 G_{27}&1&1,19,25\\
       &2&5,17,23\\
       &7&7,13,25\\
       &11&5,11,29.
\end{array}
\end{equation}
The coprime residues are covered by~\cref{c:thm:GG}. For every other
listed pair $(p,U)$ write $p=\ell^a b$, where $\ell=3$ for $G_4,G_6$,
$\ell=5$ for $H_3,G_{27}$, and $\ell\nmid b$. Then $\gcd(b,h)=1$.
By \cref{c:lem:composition}, we may compose the degree-$b$ hsop
with $a$ suitably twisted seeds from \cref{c:lem:seeds}.

For completeness, the twist arithmetic is small. For $G_4$ and $H_3$
the seed set is the entire two-element twist group. For $G_6$ the seed
set is $\{7,11\}$ modulo $12$; its even positive powers give $\{1,5\}$
and its odd powers give $\{7,11\}$. For $G_{27}$ the character is fixed
by $\{1,4\}\subset(\Z/15\Z)^\times$. In the quotient twist group,
the seed set $\{2,11\}$ has even positive powers $\{1,7\}$ and odd
powers $\{2,11\}$. Multiplication by the twist of the degree-$b$ hsop
gives exactly~\cref{c:eq:residual-table}.
Thus composition treats all degrees, including arbitrarily large powers
of $3$ or $5$.
\end{proof}

\begin{example}\label{c:ex:h3-composition}
Continue \cref{ex:h3-two-types}, with $\sigma(\sqrt5)=-\sqrt5$.
The degree-three hsop in~\cref{c:thm:GG} gives a map
$F_3:V\to V^\sigma$. Let $F_5:V\to V$ and $G_5:V\to V^\sigma$
be the two quintic maps in~\cref{c:lem:seeds}.
Applying $\sigma$ to their coefficients gives
$F_5^\sigma:V^\sigma\to V^\sigma$ and
$G_5^\sigma:V^\sigma\to V$, since $\sigma^2$ fixes $K_{H_3}$.
Thus $F_5^\sigma\circ F_3$ and $G_5^\sigma\circ F_3$ are equivariant
maps of degree $15$, with targets $V^\sigma$ and $V$, respectively.
Each composite vanishes only at the origin by~\cref{c:lem:composition}.
Their coordinate spans therefore realize both $(15,(V^*)^\sigma)$
and $(15,V^*)$. Both quotients have dimension $15^3=3375$
by~\cref{prop:koszul}; the untwisted quotient is a permutation module
and the twisted quotient is not, by~\cref{c:thm:untwisted}.
\end{example}

\subsubsection{Verification of the seeds}

\begin{proof}[Proof of \cref{c:lem:seeds}]
\emph{The cubic seeds for $G_4$ and $G_6$.}
For $G_4,G_6$ the maps are
\[
\begin{array}{c|c|c}
 W&k&F(x,y)\\\hline
 G_4&1&(x^3+4y^3,-3x^2y)\\
 G_4&2&(x^3-2y^3,-3xy^2)\\
 G_6&7&(x^3-4y^3,-3x^2y)\\
 G_6&11&(x^3+2y^3,3xy^2).
\end{array}
\]
Their only common zero is the origin. For these maps and the two
$G_{27}$ maps below, we use source generators $g=M_g^{-\mathsf T}$
from CHEVIE's matrices.

\emph{The quintic seeds for $H_3$.}
For $H_3$, put $a^2=5$, $\phi\defeq(1+a)/2$, and use the generators
\[
 s_1\defeq\begin{pmatrix}-1&0&0\\0&1&0\\0&0&1\end{pmatrix},\qquad
 s_2\defeq\begin{pmatrix}0&1&0\\0&0&1\\1&0&0\end{pmatrix},\qquad
 s_3\defeq\frac14\begin{pmatrix}
 2&-1-a&1-a\\
 -1-a&1-a&-2\\
 1-a&-2&1+a
 \end{pmatrix}.
\]
They preserve the $30$ roots consisting of $\pm e_i$ and the cyclic
permutations of $\tfrac12(\pm1,\pm\phi,\pm(\phi-1))$, and generate
a group of order $120$. Define
\[
 f\defeq x\bigl(x^4+bz^4+cx^2y^2+dx^2z^2+ey^2z^2\bigr),
 \qquad F\defeq(f(x,y,z),f(y,z,x),f(z,x,y)).
\]
Use the coefficients
\[
\begin{array}{c|rrrr}
 \sigma(a)&b&c&d&e\\\hline
 a&(a+5)/2&a+5&0&-2a\\
 -a&-3a&(3a+5)/2&-5&(9a-15)/2.
\end{array}
\]
Direct substitution proves $F(gv)=\sigma(g)F(v)$ for the three generators.
The nontrivial twist can be represented by either $k=2$ or $k=3$
modulo $5$.

\emph{The quintic seeds for $G_{27}$.}\label{sec:g27-seeds}
For $G_{27}$, in the same matrix convention, take
$F_2\defeq(f_1,f_2,f_3)$, where
\begin{align*}
 f_1&\defeq x^5+2x^3y^2-3xy^4-6x^3yz+6xy^3z
          -3x^3z^2-9xy^2z^2-18xyz^3,\\
 f_2&\defeq -3x^4y+2x^2y^3+y^5-3x^4z+3y^4z
          +9x^2yz^2+3y^3z^2-9x^2z^3-9y^2z^3,\\
 f_3&\defeq -2x^4y-4x^2y^3-\tfrac25y^5+3x^4z-6x^2y^2z+3y^4z
          -3x^2z^3+3y^2z^3+\tfrac{18}5z^5.
\end{align*}
For the second twist take $F_{11}\defeq(g_1,g_2,g_3)$, where
\begin{align*}
 g_1&\defeq x^5-2x^3y^2+xy^4+6x^3yz+6xy^3z
          +3x^3z^2-3xy^2z^2-\tfrac92xz^4,\\
 g_2&\defeq x^4y-2x^2y^3+y^5-\tfrac32x^4z-9x^2y^2z-\tfrac32y^4z
          +3x^2yz^2-3y^3z^2-\tfrac92yz^4,\\
 g_3&\defeq -x^4y-2x^2y^3-\tfrac15y^5-x^4z+2x^2y^2z-y^4z
          +6x^2z^3-6y^2z^3+\tfrac95z^5.
\end{align*}
The coordinates lie in $\Q[x,y,z]$, although their covariance uses
the nontrivial actions of $\sigma_2$ and $\sigma_{11}$ on
$\Q(\zeta_{15})$.

For all eight maps, the package~\cite{Computations} checks homogeneity
and verifies $F(gv)=\sigma(g)F(v)$ for every source generator. The calculations use
$\Q(a)$ for $H_3$ and the corresponding cyclotomic fields for
$G_4,G_6,G_{27}$. Gr\"obner bases of the coordinate ideals give
quotient dimensions $9$ for the four cubics and $125$ for the four
quintics. These quotients are finite-dimensional, so homogeneity
implies that each map vanishes only at the origin. For $H_3$, the
package also verifies the $30$ roots and group order $120$.
For the six cyclotomic maps, it also compares the printed coefficients
with the stored data and retains the covariant-space searches that
produced the maps.
\end{proof}

\subsection{Restriction constructions}\label{c:sec:restrictions}

The remaining non-well-generated exceptional groups follow from
two restriction arguments. In rank two we extend known hsop maps
to scalar overgroups; for $G_{31}$ we restrict an $E_8$ map to a
regular eigenspace.

\subsubsection{The non-well-generated rank-two groups}\label{c:sec:scalar}

We extend the covariance of known hsop maps by scalar homogeneity,
then restrict to the seven non-well-generated rank-two groups.

\begin{lemma}\label{c:lem:scalar}
Let $H\subseteq\mathrm{GL}(V)$ be an irreducible well-generated complex
reflection group with Coxeter number $h$, and let $M$ be a positive
multiple of $h$. If $p>0$ and $\gcd(p,M)=1$, then
every subgroup $W\subseteq\mu_M H$ has a degree-$p$ hsop carrying
$(V^*)^{\sigma_p}$, where $\sigma_p(\zeta_M)=\zeta_M^p$.
\end{lemma}

\begin{proof}
Choose the degree-$p$ map from~\cref{c:thm:GG}. Homogeneity gives
$F(\lambda v)=\lambda^pF(v)=\sigma_p(\lambda)F(v)$ for
$\lambda\in\mu_M$. Thus its covariance extends from $H$ to
$\mu_M H$, and then restricts to $W$. Its zero locus is unchanged.
\end{proof}

Apply the lemma to the following groups:
\begin{equation}\label{c:eq:scalar-table}
\begin{array}{c|c|c|c}
 W&H&M&\text{common projective image}\\\hline
 G_7&G_4&12&A_4\\
 G_{11},G_{12},G_{13},G_{15}&G_8&24&S_4\\
 G_{19},G_{22}&G_{16}&60&A_5.
\end{array}
\end{equation}
Here the required inclusions follow directly from the classification of
primitive finite subgroups of $\mathrm{GL}_2(\C)$ by
Shephard--Todd~\cite{ST}. After
conjugating their projective images to agree, write each $w\in W$ as
$\lambda h$ with $h\in H$. In each row the determinant characters of
both groups have order dividing $M/2$. Therefore
$\lambda^2=\det(w)/\det(h)\in\mu_{M/2}$ and $\lambda\in\mu_M$.

For each group in~\cref{c:eq:scalar-table}, the calculation in
\cref{c:prop:exceptional-polynomiality} shows that condition (3) of \cref{c:thm:main}
is equivalent to $\gcd(p,M)=1$ and $U\cong(V^*)^{\sigma_p}$.
\Cref{c:lem:scalar} therefore proves sufficiency in every degree.
The package~\cite{Computations} checks the determinant and projective
orders and compares the degree residues of posh pairs over a full period; the
inclusions themselves are justified by the preceding argument.

\subsubsection{Restriction from \texorpdfstring{$E_8$ to $G_{31}$}{E8 to G31}}\label{c:sec:eigenspace}

The last exceptional group is obtained by restricting an $E_8$ hsop
to a regular eigenspace. Homogeneity determines the target eigenspace
and therefore its Galois twist.

\begin{lemma}\label{c:lem:eigenspace}
Suppose a finite group $\widetilde W$ acts on a rational vector space
$\widetilde V_{\Q}$, and put
$\widetilde V\defeq\C\otimes_{\Q}\widetilde V_{\Q}$.
Let $g\in\widetilde W$ have a root-of-unity eigenvalue $\zeta$ on
$\widetilde V$, and put $V\defeq\ker(g-\zeta)$ and
$W\defeq C_{\widetilde W}(g)$. If $p>0$ and
$\gcd(p,\operatorname{ord}\zeta)=1$, then a homogeneous degree-$p$
$\widetilde W$-equivariant map $F:\widetilde V\to\widetilde V$
with $F^{-1}(0)=\{0\}$ restricts to
\[
 V\longrightarrow \ker(g-\zeta^p)\cong V^{\sigma_p},
 \qquad \sigma_p(\zeta)=\zeta^p,
\]
and its coordinate functions span an hsop for the action of $W$ on $V$.
\end{lemma}

\begin{proof}
For $v\in V$, equivariance and homogeneity give
$gF(v)=F(gv)=\zeta^pF(v)$. The restricted map has no nonzero zero.
Rationality identifies the two eigenspaces by Galois conjugation, with
equal dimension and the asserted action of the centralizer.
The height bound then forces the coordinate functions of the restricted
map to be linearly independent, so their span has dimension $\dim V$.
\end{proof}

The group $G_{31}$ is the centralizer of a regular element of order $4$
in $E_8$, acting on its $i$-eigenspace~\cite[Theorem~1.9(2) and Example~1.10]{Bessis}.
The invariant degrees on that eigenspace are $8,12,20,24$. For every
$p$ coprime to $30$, the rational reflection representation of $E_8$
has an untwisted degree-$p$ hsop by~\cref{c:thm:GG}. Since $p$ is odd,
\Cref{c:lem:eigenspace} gives a $G_{31}$ hsop with trivial twist
if $p\equiv1\pmod4$, and conjugate twist if $p\equiv3\pmod4$.
These are precisely the two generator types and their degrees for
$G_{31}$ in \cref{c:tab:exceptional}.

The package~\cite{Computations} checks the degree selection by
divisibility by four and the two twist congruences. This is an
arithmetic check of the restriction argument.

\subsection{Completion and specializations}\label{c:sec:completion}

\begin{proof}[Completion of the proof of \cref{c:thm:main}]
\Cref{c:prop:reduction} proves $(1)\Rightarrow(2)\Rightarrow(3)$
and reduces (3) to Galois twists of $V^*$.
\Cref{c:prop:infinite-classification,c:prop:exceptional-polynomiality}
then give $(3)\Rightarrow(4)$.

For $(4)\Rightarrow(1)$, \cref{c:prop:infinite-classification}
proves that the listed infinite-family pairs are posh. For well-generated
exceptional groups with $\gcd(p,h)=1$, the listed type is
$U\cong(V^*)^{\sigma_p}$ and \cref{c:thm:GG} supplies the hsop.
\Cref{c:prop:seed-coverage} treats all remaining pairs for
$G_4,G_6,H_3,G_{27}$.
Finally, the applications of \cref{c:lem:scalar,c:lem:eigenspace}
in \cref{c:sec:restrictions} give every listed pair for the
non-well-generated exceptional groups. These cases exhaust the
tables and complete the proof.
\end{proof}

\coprimeproposition*
\begin{proof}
\Cref{c:thm:GG} proves that $(p,(V^*)^{\sigma_p})$ is posh. Conversely, the
well-generated rows of \cref{tab:infinite-algebraic,c:tab:exceptional}
in \cref{c:thm:main} allow only $U\cong(V^*)^{\sigma_p}$ when
$\gcd(p,h)=1$.
\end{proof}

\begingroup
\renewcommand{\footnote}[1]{}
\coxeterpermutationcorollary*
\endgroup
\begin{proof}
Let $\sigma_p$ be the automorphism of $\Q(\zeta_h)$ defined by
$\sigma_p(\zeta_h)=\zeta_h^p$. By \cref{prop:coprime-hsop},
$\Theta\cong(V^*)^{\sigma_p}$, so \cref{c:thm:untwisted} gives
\[
 A_\Theta\text{ is a permutation module}
 \quad\Longleftrightarrow\quad
 (V^*)^{\sigma_p}\cong V^*
 \quad\Longleftrightarrow\quad
 \sigma_p|_{K_W}=\mathrm{id}.
\]
The last equivalence follows because the character values of $V^*$
generate $K_W$. By V.~Reiner, V.~Ripoll, and C.~Stump~\cite[Theorem~1.8(iv)]{RRS},
the coefficients of the characteristic polynomial of $c$ on $V$
also generate $K_W$. Since $c$ has order $h$, applying $\sigma_p$
to these coefficients gives the characteristic polynomial of $c^p$.
Thus $\sigma_p|_{K_W}=\mathrm{id}$ exactly when $c$ and $c^p$ have
the same eigenvalues, with multiplicities. Because $\gcd(p,h)=1$,
both are regular of order $h$, and this equality is equivalent to
conjugacy in $W$ by Reiner--Ripoll--Stump~\cite[Lemma~3.4 and its proof]{RRS}.

Choose a primitive $h$th root $\zeta$ for which $c$ is $\zeta$-regular.
By Springer's eigenvalue theorem, as recalled in
Reiner--Ripoll--Stump~\cite[Theorem~2.4(ii)]{RRS}, the eigenvalues of
$c$ on $V$ are $\zeta^{-e_1},\ldots,\zeta^{-e_r}$. Thus $c^p$ has the
same eigenvalues exactly when the multisets of residues $p e_i$ and
$e_i$ modulo $h$ agree. Since $0\leq e_i<h$, this is the displayed
criterion.
\end{proof}

Specializing the classification to $U=V^*$ and applying
\cref{c:thm:untwisted} gives
\cref{tab:infinite-untwisted,tab:exceptional-untwisted}. In the infinite
families the degrees are $\gcd(p,n)=1$ for $S_n$, $p\equiv1\pmod m$
for non-dihedral $G(m,d,r)$, and $p\equiv\pm1\pmod m$ for $I_2(m)$.

\section{Further questions}
\label{sec:questions}\label{sec:beyond-hsops}

\subsection{Capricious hsops}

It seems natural to ask what happens if we relax the conditions in~\Cref{def:parking_hsop} to allow a graded $W$-stable subspace
$\Theta=\bigoplus_{a>0}\Theta_a\subseteq S$ of dimension $r$, and
require $S/(\Theta)$ to be finite-dimensional. We call such a $\Theta$ a
\defn{capricious hsop}, and its quotient $S/(\Theta)$ a
\defn{capricious parking space}~\cite[Section~6, p.~1274]{KonheimWeiss}.

A capricious hsop may be a nonfaithful representation of $W$, and its
homogeneous generators may have different degrees. For irreducible $W$, posh hsops are
precisely the faithful capricious hsops: the height argument of
\cref{a:thm:automatic} also applies to generators of different degrees,
so \cref{a:prop:faithful-reflection-module} makes a faithful $\Theta$
a Galois twist of $V^*$; a graded copy must lie in one degree.

\begin{problem}\label{prob:nonfaithful-hsops}
For an irreducible complex reflection group $W$, determine which
nonfaithful graded $W$-modules occur as capricious hsops,
and characterize their possible degrees.
\end{problem}


Several constructions illustrate the importance of this extension.
\begin{itemize}
\item Let $g_1,\ldots,g_r\in S^W$ be homogeneous invariant parameters
of degrees $a_1,\ldots,a_r$, and put
$B\defeq S^W/(g_1,\ldots,g_r)$. Tensoring the decomposition
in~\cref{b:eq:coinvariant-decomposition} with $B$ gives
\[
 S/(g_1,\ldots,g_r)\cong C_W\otimes B,\qquad
 \Hilb(B;q)=\prod_{i=1}^r\frac{[a_i]_q}{[d_i]_q}\in\N[q].
\]
The coefficient of $q^j$ in $\Hilb(B;q)$ counts copies of $C_W$
whose grading is shifted by $j$. After forgetting grading, the
quotient is a sum of $\prod_i a_i/|W|$ regular representations.
The construction in~\cref{b:prop:invariant-parameters} supplies such
capricious hsops of degrees $a_i=p+e_i(\Theta^*)$ from any
$W$-stable hsop $\Theta\subseteq S_p$.
\item Taking $g_i=f_i$, the fundamental invariants, gives $B=\C$
and recovers the coinvariant algebra $C_W=S/(f_1,\ldots,f_r)$
in~\cref{b:sec:coinvariants}. The regular action of $W$ on itself
gives an ungraded combinatorial model for this quotient.
\item If $W$ is irreducible and well-generated, choose the top degree
invariant ($\deg f_r=h$) so that the reduced branch discriminant
is monic and centered in $f_r$, say
\[
 \Delta(f_1,\ldots,f_r)
 =f_r^r+\alpha_2f_r^{r-2}+\cdots+\alpha_r,
 \qquad \alpha_i\in\C[f_1,\ldots,f_{r-1}].
\]
Finiteness of the Lyashko--Looijenga map makes
$f_r,\alpha_2,\ldots,\alpha_r$ invariant parameters of degrees
$h,2h,\ldots,rh$; see~\cite[Section~7.2]{EquivariantParameters}
and~\cite[Lemma~4.4]{DouvropoulosLL}.
Here $\Hilb(B;q)=\prod_{i=1}^r[ih]_q/[d_i]_q$.
The invariant subspace of the quotient has dimension $r!h^r/|W|$,
the number of reduced reflection factorizations of a Coxeter element $c$.
These factorizations index a generic fiber of the Lyashko--Looijenga
map~\cite[Theorem~7.20]{Bessis}.
\item A final motivation comes from Griffeth's work on diagonal coinvariants.
$V\oplus V^*$ is the underlying complex representation of the
quaternionification of $V$, and the diagonal coinvariant ring
\[
 D_W\defeq\C[V\oplus V^*]\big/
 \big\langle\bigl(\C[V\oplus V^*]^W\bigr)_+\big\rangle
\]
records the scheme-theoretic zero fiber of the quotient by $W$.
For irreducible $W$, put $g=2N/r$, where $N$ is the number of reflections;
this is an integer. Griffeth's construction from rational Cherednik
algebras gives an hsop of degree $g+1$ carrying the irreducible
KZ twist $\kappa(V^*)$, where $\kappa$ is the representation permutation
in that construction~\cite[Theorem~1.2 and Section~3R]{GriffethDiagonal}.
Its quotient $L$ has dimension $(g+1)^r$. After tensoring by
$\det(V)^*$, the associated graded module of $L$ for the filtration
generated by its unique determinant copy is a quotient of
$D_W$~\cite[Lemma~3.1]{GriffethDiagonal}.
These generators need not be faithful: for $W=G(6,3,r)$, $r\geq2$,
the degree is $6r-3$ and
\[
 \kappa(V^*)\cong\operatorname{span}_{\C}\{x_1^{6r-3},\ldots,x_r^{6r-3}\},
 \qquad \ker\kappa(V^*)\cong(\Z/3\Z)^{r-1};
\]
see~\cite[Section~8]{EquivariantParameters}.
\end{itemize}

\subsection{Polynomial differential forms}

A second generalization retains the posh condition but adds differential
information. For a posh hsop $\Theta\subseteq S_p$, define
\[
 \Omega_\Theta\defeq
 \bigl(S\otimes\textstyle\bigwedge V^*\bigr)/(\Theta,d\Theta),
 \qquad
 d\theta=\sum_{i=1}^r
       \frac{\partial\theta}{\partial x_i}\,dx_i.
\]
Here $x_1,\ldots,x_r$ is a basis of $V^*$, and the exterior generators
$dx_i$ transform in another copy of $V^*$.
Give $x_i$ bidegree $(1,0)$ and $dx_i$ bidegree $(0,1)$, so that $\theta$ and $d\theta$ have bidegrees $(p,0)$ and
$(p-1,1)$, respectively. 
The relations $d\theta=0$ give the K\"ahler differentials of
$A_\Theta$~\cite[Tag~00RU]{Stacks}: a differential form on the quotient
must respect the defining relations $\theta=0$. Thus
\[
 \Omega_\Theta\cong
 \textstyle\bigwedge_{A_\Theta}\Omega^1_{A_\Theta/\C}.
\]
Its component of exterior degree zero is $A_\Theta$. The same construction for the capricious hsop of
fundamental invariants gives the usual super coinvariant
algebra~\cite[Section~5]{SwansonWallachForms}.

\begin{problem}\label{prob:differential-parking}
For a fixed posh pair $(p,U)$, determine the bigraded $W$-characters
of $\Omega_\Theta$ as $\Theta$ varies over the degree-$p$ posh hsops
carrying $U$. Characterize those $\Theta$ for which $\Omega_\Theta$,
after forgetting both gradings, is a permutation representation.
\end{problem}

Unlike the character of $A_\Theta$, the character of $\Omega_\Theta$
can depend on the choice of hsop within a fixed posh pair.

\begin{example}\label{ex:b2-differential-parking}
Let $W=B_2$ act on $\C[x,y]$ by signed permutations, and put
\[
 \Theta_a\defeq\operatorname{span}_{\C}
 \{x^3+axy^2,\ y^3+ax^2y\},
 \qquad a\in\C\setminus\{1,-1\}.
\]
These are posh hsops carrying $V^*$. Put $\Omega_a\defeq\Omega_{\Theta_a}$,
and write $\Omega_a^j$ for its component of exterior degree $j$.
Its top component is
\[
 \Omega_a^2\cong
 \frac{\C[x,y]}{(3x^2+ay^2,\ 2axy,\ 3y^2+ax^2)}\,dx\wedge dy.
\]
Euler's identity shows that the cubic generators already lie in the
displayed ideal. For $a\notin\{0,\pm3\}$, the displayed quotient of $\C[x,y]$ has basis
$1,x,y$. At $a=0,3,-3$, it has one additional class represented by
$xy$, $x^2-y^2$, and $x^2+y^2$, respectively.

The exterior derivative and contraction with the Euler vector field
$x\partial/\partial x+y\partial/\partial y$ descend to $\Omega_a$:
the derivative sends $\theta$ to $d\theta$ and kills $d\theta$,
whereas contraction kills $\theta$ and sends $d\theta$ to $3\theta$.
Both commute with $W$. The sum of their composites in the two orders
acts on bidegree $(k,j)$ by multiplication by $k+j$. The resulting
de Rham complex is therefore
exact in positive total degree and has only $\C$ in total degree zero.
Since $\Omega_a^0=A_{\Theta_a}$, its alternating character gives
\[
 [\Omega_a]=2[A_{\Theta_a}]+2[\Omega_a^2]-[\C]
\]
in the representation ring of $W$, after forgetting both gradings.
Combining the signed-permutation actions on the displayed basis
classes, including the determinant action on $dx\wedge dy$, with
the quotient character in~\cref{prop:koszul} gives the following table.
Let $c$ be a quarter-turn, $s$ an axis reflection, and $t$ a diagonal
reflection. The ordinary ungraded characters are
\[
 \begin{array}{c|rrrrr}
 a & 1 & c^2 & c & s & t\\\hline
 a\notin\{0,\pm1,\pm3\} & 23 & -1 & 3 & 3 & 3\\
 0 & 25 & 1 & 1 & 5 & 1\\
 3 & 25 & 1 & 1 & 1 & 5\\
 -3 & 25 & 1 & 5 & 1 & 1
 \end{array}
\]
$\Omega_a$ is a permutation module exactly for $a=0$ and $a=3$.  After forgetting both gradings,
\[
 \begin{aligned}
 \Omega_0&\cong\C\oplus\C[W]^{\oplus2}
       \oplus\C[W/\langle s\rangle]^{\oplus2},\\
 \Omega_3&\cong\C\oplus\C[W]^{\oplus2}
       \oplus\C[W/\langle t\rangle]^{\oplus2}.
 \end{aligned}
\]
For all other admissible parameters, a permutation character is
impossible: the generic character takes value $-1$ at $c^2$, and
at $a=-3$ its values $5$ at $c$ and $1$ at $c^2$ contradict
$\Fix(c)\subseteq\Fix(c^2)$ for a permutation action.
\end{example}

\section*{Acknowledgements}
We thank Minh-T\^am Trinh, Stephen Griffeth, Theo Douvropoulos, and Carlos Arreche for many
helpful discussions and suggestions, and the University of Texas at Dallas for providing inspiration through its fashionable multi-tiered parking permit system. Both authors were partially supported by
NSF grant DMS-2246877. We are grateful to OpenAI for support through an Academic Workspace; although our original code from 2024 constructing hsops was written by hand in Mathematica and GAP3, for this paper we used various GPT models for exploration, for computation, and for writing.  We take full responsibility for all wordplay.

\appendix
\section{Classification tables}\label{sec:tables}

The tables list all posh pairs and their untwisted specializations
from \cref{c:thm:main,c:thm:untwisted}. 
Recall that $p$ is a positive integer, $\mathfrak f_W$ is the conductor of the reflection character
field, and $k$ specifies the automorphism
$\sigma_k(\zeta_{\mathfrak f_W})=\zeta_{\mathfrak f_W}^k$ (where
$\zeta_{\mathfrak f_W}\defeq\exp(2\pi i/\mathfrak f_W)$ and
$\gcd(k,\mathfrak f_W)=1$). We list one $k$ for each distinct twisted
character. The integer $M$ is the least period of the degrees $p$
for which $(p,(V^*)^{\sigma_k})$ is posh: an entry $R$
in the column $p\bmod M$ means exactly the positive degrees
$p\equiv s\pmod M$ with $s\in R$. 

\begin{table}[htbp]
\centering
\begin{tabular}{@{}L{.14\textwidth}L{.29\textwidth}L{.30\textwidth}L{.19\textwidth}@{}}
\toprule
Group & Family parameters & Degree $p$ and module $U$ & Construction\\
\midrule
$S_n$ & $n\ge2$, $r=n-1$ & $\gcd(p,n)=1$,\newline $U\cong V^*$ & \cref{c:ex:haiman}\\
\addlinespace
$G(m,d,r)$ & $m\ge2$, $d\mid m$, $r\ge1$;\newline $d=1$ if $r=1$;\newline $(m,d,r)\ne(2,2,2)$ & $\gcd(p,m)=1$,\newline $U\cong(V^*)^{\sigma_p}$ & \cref{c:sec:imprimitive}\\
\bottomrule
\end{tabular}
\caption{Posh pairs in the infinite families, for \cref{c:thm:main}.} 
\label{tab:infinite-algebraic}
\end{table}

\begin{table}[htbp]
\centering
\begin{tabular}{@{}L{.17\textwidth}L{.35\textwidth}L{.20\textwidth}L{.20\textwidth}@{}}
\toprule
Group & Family parameters & Degree $p$ & Construction\\
\midrule
$S_n$ & $n\ge2$, $r=n-1$ & $\gcd(p,n)=1$ & \cref{c:ex:haiman}\\
\addlinespace
$G(m,d,r)$ & $m\ge2$, $d\mid m$, $r\ge1$;\newline $d=1$ if $r=1$;\newline $(d,r)\ne(m,2)$ & $p\equiv1\pmod m$ & \cref{c:sec:imprimitive}\\
\addlinespace
$G(m,m,2)$ & $m\ge3$, $r=2$ & $p\equiv\pm1\pmod m$ & \cref{c:sec:imprimitive}\\
\bottomrule
\end{tabular}
\caption{Untwisted posh pairs $(p,V^*)$ in the infinite
families, for \cref{c:thm:untwisted}. 
The degree conditions agree with \cite[Table 1]{IO}.}
\label{tab:infinite-untwisted}
\end{table}

\clearpage
\begingroup\small
\begingroup\setlength{\tabcolsep}{3pt}
\begin{longtable}{@{}lL{.24\textwidth}rrrrL{.26\textwidth}L{.14\textwidth}@{}}
\toprule $W$ & Degrees & $h$ & $\mathfrak f_W$ & $k$ & Period $M$ & $p\bmod M$ & Construction\\\midrule
\endfirsthead
\toprule $W$ & Degrees & $h$ & $\mathfrak f_W$ & $k$ & Period $M$ & $p\bmod M$ & Construction\\\midrule
\endhead
\bottomrule\endfoot
\endlastfoot
$G_{4}$ & $4,\allowbreak 6$ & 6 & 3 & 1 & 6 & $\{1,\allowbreak 3\}$ & \cref{c:sec:noncoprime} \\*
 &  &  &  & 2 &  & $\{3,\allowbreak 5\}$ &  \\
\arrayrulecolor{black!25}\midrule[.25pt]\arrayrulecolor{black}
$G_{5}$ & $6,\allowbreak 12$ & 12 & 3 & 1 & 6 & $\{1\}$ & \cref{c:thm:GG} \\*
 &  &  &  & 2 &  & $\{5\}$ &  \\
\arrayrulecolor{black!25}\midrule[.25pt]\arrayrulecolor{black}
$G_{6}$ & $4,\allowbreak 12$ & 12 & 12 & 1 & 12 & $\{1,\allowbreak 9\}$ & \cref{c:sec:noncoprime} \\*
 &  &  &  & 5 &  & $\{5,\allowbreak 9\}$ &  \\*
 &  &  &  & 7 &  & $\{3,\allowbreak 7\}$ &  \\*
 &  &  &  & 11 &  & $\{3,\allowbreak 11\}$ &  \\
\arrayrulecolor{black!25}\midrule[.25pt]\arrayrulecolor{black}
$G_{7}$ & $12,\allowbreak 12$ & -- & 12 & 1 & 12 & $\{1\}$ & \cref{c:lem:scalar} \\*
 &  &  &  & 5 &  & $\{5\}$ &  \\*
 &  &  &  & 7 &  & $\{7\}$ &  \\*
 &  &  &  & 11 &  & $\{11\}$ &  \\
\arrayrulecolor{black!25}\midrule[.25pt]\arrayrulecolor{black}
$G_{8}$ & $8,\allowbreak 12$ & 12 & 4 & 1 & 12 & $\{1,\allowbreak 5\}$ & \cref{c:thm:GG} \\*
 &  &  &  & 3 &  & $\{7,\allowbreak 11\}$ &  \\
\arrayrulecolor{black!25}\midrule[.25pt]\arrayrulecolor{black}
$G_{9}$ & $8,\allowbreak 24$ & 24 & 8 & 1 & 24 & $\{1,\allowbreak 17\}$ & \cref{c:thm:GG} \\*
 &  &  &  & 3 &  & $\{11,\allowbreak 19\}$ &  \\*
 &  &  &  & 5 &  & $\{5,\allowbreak 13\}$ &  \\*
 &  &  &  & 7 &  & $\{7,\allowbreak 23\}$ &  \\
\arrayrulecolor{black!25}\midrule[.25pt]\arrayrulecolor{black}
$G_{10}$ & $12,\allowbreak 24$ & 24 & 12 & 1 & 12 & $\{1\}$ & \cref{c:thm:GG} \\*
 &  &  &  & 5 &  & $\{5\}$ &  \\*
 &  &  &  & 7 &  & $\{7\}$ &  \\*
 &  &  &  & 11 &  & $\{11\}$ &  \\
\arrayrulecolor{black!25}\midrule[.25pt]\arrayrulecolor{black}
$G_{11}$ & $24,\allowbreak 24$ & -- & 24 & 1 & 24 & $\{1\}$ & \cref{c:lem:scalar} \\*
 &  &  &  & 5 &  & $\{5\}$ &  \\*
 &  &  &  & 7 &  & $\{7\}$ &  \\*
 &  &  &  & 11 &  & $\{11\}$ &  \\*
 &  &  &  & 13 &  & $\{13\}$ &  \\*
 &  &  &  & 17 &  & $\{17\}$ &  \\*
 &  &  &  & 19 &  & $\{19\}$ &  \\*
 &  &  &  & 23 &  & $\{23\}$ &  \\
\arrayrulecolor{black!25}\midrule[.25pt]\arrayrulecolor{black}
$G_{12}$ & $6,\allowbreak 8$ & -- & 8 & 1 & 24 & $\{1,\allowbreak 11,\allowbreak 17,\allowbreak 19\}$ & \cref{c:lem:scalar} \\*
 &  &  &  & 5 &  & $\{5,\allowbreak 7,\allowbreak 13,\allowbreak 23\}$ &  \\
\arrayrulecolor{black!25}\midrule[.25pt]\arrayrulecolor{black}
$G_{13}$ & $8,\allowbreak 12$ & -- & 8 & 1 & 24 & $\{1,\allowbreak 17\}$ & \cref{c:lem:scalar} \\*
 &  &  &  & 3 &  & $\{11,\allowbreak 19\}$ &  \\*
 &  &  &  & 5 &  & $\{5,\allowbreak 13\}$ &  \\*
 &  &  &  & 7 &  & $\{7,\allowbreak 23\}$ &  \\
\arrayrulecolor{black!25}\midrule[.25pt]\arrayrulecolor{black}
$G_{14}$ & $6,\allowbreak 24$ & 24 & 24 & 1 & 24 & $\{1,\allowbreak 19\}$ & \cref{c:thm:GG} \\*
 &  &  &  & 5 &  & $\{5,\allowbreak 23\}$ &  \\*
 &  &  &  & 7 &  & $\{7,\allowbreak 13\}$ &  \\*
 &  &  &  & 11 &  & $\{11,\allowbreak 17\}$ &  \\
\arrayrulecolor{black!25}\midrule[.25pt]\arrayrulecolor{black}
$G_{15}$ & $12,\allowbreak 24$ & -- & 24 & 1 & 24 & $\{1\}$ & \cref{c:lem:scalar} \\*
 &  &  &  & 5 &  & $\{5\}$ &  \\*
 &  &  &  & 7 &  & $\{7\}$ &  \\*
 &  &  &  & 11 &  & $\{11\}$ &  \\*
 &  &  &  & 13 &  & $\{13\}$ &  \\*
 &  &  &  & 17 &  & $\{17\}$ &  \\*
 &  &  &  & 19 &  & $\{19\}$ &  \\*
 &  &  &  & 23 &  & $\{23\}$ &  \\
\arrayrulecolor{black!25}\midrule[.25pt]\arrayrulecolor{black}
$G_{16}$ & $20,\allowbreak 30$ & 30 & 5 & 1 & 30 & $\{1,\allowbreak 11\}$ & \cref{c:thm:GG} \\*
 &  &  &  & 2 &  & $\{7,\allowbreak 17\}$ &  \\*
 &  &  &  & 3 &  & $\{13,\allowbreak 23\}$ &  \\*
 &  &  &  & 4 &  & $\{19,\allowbreak 29\}$ &  \\
\arrayrulecolor{black!25}\midrule[.25pt]\arrayrulecolor{black}
$G_{17}$ & $20,\allowbreak 60$ & 60 & 20 & 1 & 60 & $\{1,\allowbreak 41\}$ & \cref{c:thm:GG} \\*
 &  &  &  & 3 &  & $\{23,\allowbreak 43\}$ &  \\*
 &  &  &  & 7 &  & $\{7,\allowbreak 47\}$ &  \\*
 &  &  &  & 9 &  & $\{29,\allowbreak 49\}$ &  \\*
 &  &  &  & 11 &  & $\{11,\allowbreak 31\}$ &  \\*
 &  &  &  & 13 &  & $\{13,\allowbreak 53\}$ &  \\*
 &  &  &  & 17 &  & $\{17,\allowbreak 37\}$ &  \\*
 &  &  &  & 19 &  & $\{19,\allowbreak 59\}$ &  \\
\arrayrulecolor{black!25}\midrule[.25pt]\arrayrulecolor{black}
$G_{18}$ & $30,\allowbreak 60$ & 60 & 15 & 1 & 30 & $\{1\}$ & \cref{c:thm:GG} \\*
 &  &  &  & 2 &  & $\{17\}$ &  \\*
 &  &  &  & 4 &  & $\{19\}$ &  \\*
 &  &  &  & 7 &  & $\{7\}$ &  \\*
 &  &  &  & 8 &  & $\{23\}$ &  \\*
 &  &  &  & 11 &  & $\{11\}$ &  \\*
 &  &  &  & 13 &  & $\{13\}$ &  \\*
 &  &  &  & 14 &  & $\{29\}$ &  \\
\arrayrulecolor{black!25}\midrule[.25pt]\arrayrulecolor{black}
$G_{19}$ & $60,\allowbreak 60$ & -- & 60 & 1 & 60 & $\{1\}$ & \cref{c:lem:scalar} \\*
 &  &  &  & 7 &  & $\{7\}$ &  \\*
 &  &  &  & 11 &  & $\{11\}$ &  \\*
 &  &  &  & 13 &  & $\{13\}$ &  \\*
 &  &  &  & 17 &  & $\{17\}$ &  \\*
 &  &  &  & 19 &  & $\{19\}$ &  \\*
 &  &  &  & 23 &  & $\{23\}$ &  \\*
 &  &  &  & 29 &  & $\{29\}$ &  \\*
 &  &  &  & 31 &  & $\{31\}$ &  \\*
 &  &  &  & 37 &  & $\{37\}$ &  \\*
 &  &  &  & 41 &  & $\{41\}$ &  \\*
 &  &  &  & 43 &  & $\{43\}$ &  \\*
 &  &  &  & 47 &  & $\{47\}$ &  \\*
 &  &  &  & 49 &  & $\{49\}$ &  \\*
 &  &  &  & 53 &  & $\{53\}$ &  \\*
 &  &  &  & 59 &  & $\{59\}$ &  \\
\arrayrulecolor{black!25}\midrule[.25pt]\arrayrulecolor{black}
$G_{20}$ & $12,\allowbreak 30$ & 30 & 15 & 1 & 30 & $\{1,\allowbreak 19\}$ & \cref{c:thm:GG} \\*
 &  &  &  & 2 &  & $\{17,\allowbreak 23\}$ &  \\*
 &  &  &  & 7 &  & $\{7,\allowbreak 13\}$ &  \\*
 &  &  &  & 11 &  & $\{11,\allowbreak 29\}$ &  \\
\arrayrulecolor{black!25}\midrule[.25pt]\arrayrulecolor{black}
$G_{21}$ & $12,\allowbreak 60$ & 60 & 60 & 1 & 60 & $\{1,\allowbreak 49\}$ & \cref{c:thm:GG} \\*
 &  &  &  & 7 &  & $\{7,\allowbreak 43\}$ &  \\*
 &  &  &  & 11 &  & $\{11,\allowbreak 59\}$ &  \\*
 &  &  &  & 13 &  & $\{13,\allowbreak 37\}$ &  \\*
 &  &  &  & 17 &  & $\{17,\allowbreak 53\}$ &  \\*
 &  &  &  & 19 &  & $\{19,\allowbreak 31\}$ &  \\*
 &  &  &  & 23 &  & $\{23,\allowbreak 47\}$ &  \\*
 &  &  &  & 29 &  & $\{29,\allowbreak 41\}$ &  \\
\arrayrulecolor{black!25}\midrule[.25pt]\arrayrulecolor{black}
$G_{22}$ & $12,\allowbreak 20$ & -- & 20 & 1 & 60 & $\{1,\allowbreak 29,\allowbreak 41,\allowbreak 49\}$ & \cref{c:lem:scalar} \\*
 &  &  &  & 3 &  & $\{7,\allowbreak 23,\allowbreak 43,\allowbreak 47\}$ &  \\*
 &  &  &  & 11 &  & $\{11,\allowbreak 19,\allowbreak 31,\allowbreak 59\}$ &  \\*
 &  &  &  & 13 &  & $\{13,\allowbreak 17,\allowbreak 37,\allowbreak 53\}$ &  \\
\arrayrulecolor{black!25}\midrule[.25pt]\arrayrulecolor{black}
$G_{23}$ & $2,\allowbreak 6,\allowbreak 10$ & 10 & 5 & 1 & 10 & $\{1,\allowbreak 5,\allowbreak 9\}$ & \cref{c:sec:noncoprime} \\*
 &  &  &  & 2 &  & $\{3,\allowbreak 5,\allowbreak 7\}$ &  \\
\arrayrulecolor{black!25}\midrule[.25pt]\arrayrulecolor{black}
$G_{24}$ & $4,\allowbreak 6,\allowbreak 14$ & 14 & 7 & 1 & 14 & $\{1,\allowbreak 9,\allowbreak 11\}$ & \cref{c:thm:GG} \\*
 &  &  &  & 3 &  & $\{3,\allowbreak 5,\allowbreak 13\}$ &  \\
\arrayrulecolor{black!25}\midrule[.25pt]\arrayrulecolor{black}
$G_{25}$ & $6,\allowbreak 9,\allowbreak 12$ & 12 & 3 & 1 & 6 & $\{1\}$ & \cref{c:thm:GG} \\*
 &  &  &  & 2 &  & $\{5\}$ &  \\
\arrayrulecolor{black!25}\midrule[.25pt]\arrayrulecolor{black}
$G_{26}$ & $6,\allowbreak 12,\allowbreak 18$ & 18 & 3 & 1 & 6 & $\{1\}$ & \cref{c:thm:GG} \\*
 &  &  &  & 2 &  & $\{5\}$ &  \\
\arrayrulecolor{black!25}\midrule[.25pt]\arrayrulecolor{black}
$G_{27}$ & $6,\allowbreak 12,\allowbreak 30$ & 30 & 15 & 1 & 30 & $\{1,\allowbreak 19,\allowbreak 25\}$ & \cref{c:sec:noncoprime} \\*
 &  &  &  & 2 &  & $\{5,\allowbreak 17,\allowbreak 23\}$ &  \\*
 &  &  &  & 7 &  & $\{7,\allowbreak 13,\allowbreak 25\}$ &  \\*
 &  &  &  & 11 &  & $\{5,\allowbreak 11,\allowbreak 29\}$ &  \\
\arrayrulecolor{black!25}\midrule[.25pt]\arrayrulecolor{black}
$G_{28}$ & $2,\allowbreak 6,\allowbreak 8,\allowbreak 12$ & 12 & 1 & 1 & 6 & $\{1,\allowbreak 5\}$ & \cref{c:thm:GG} \\
\arrayrulecolor{black!25}\midrule[.25pt]\arrayrulecolor{black}
$G_{29}$ & $4,\allowbreak 8,\allowbreak 12,\allowbreak 20$ & 20 & 4 & 1 & 20 & $\{1,\allowbreak 9,\allowbreak 13,\allowbreak 17\}$ & \cref{c:thm:GG} \\*
 &  &  &  & 3 &  & $\{3,\allowbreak 7,\allowbreak 11,\allowbreak 19\}$ &  \\
\arrayrulecolor{black!25}\midrule[.25pt]\arrayrulecolor{black}
$G_{30}$ & $2,\allowbreak 12,\allowbreak 20,\allowbreak 30$ & 30 & 5 & 1 & 30 & $\{1,\allowbreak 11,\allowbreak 19,\allowbreak 29\}$ & \cref{c:thm:GG} \\*
 &  &  &  & 2 &  & $\{7,\allowbreak 13,\allowbreak 17,\allowbreak 23\}$ &  \\
\arrayrulecolor{black!25}\midrule[.25pt]\arrayrulecolor{black}
$G_{31}$ & $8,\allowbreak 12,\allowbreak 20,\allowbreak 24$ & -- & 4 & 1 & 60 & $\{1,\allowbreak 13,\allowbreak 17,\allowbreak 29,\allowbreak 37,\allowbreak 41,\allowbreak 49,\allowbreak 53\}$ & \cref{c:lem:eigenspace} \\*
 &  &  &  & 3 &  & $\{7,\allowbreak 11,\allowbreak 19,\allowbreak 23,\allowbreak 31,\allowbreak 43,\allowbreak 47,\allowbreak 59\}$ &  \\
\arrayrulecolor{black!25}\midrule[.25pt]\arrayrulecolor{black}
$G_{32}$ & $12,\allowbreak 18,\allowbreak 24,\allowbreak 30$ & 30 & 3 & 1 & 30 & $\{1,\allowbreak 7,\allowbreak 13,\allowbreak 19\}$ & \cref{c:thm:GG} \\*
 &  &  &  & 2 &  & $\{11,\allowbreak 17,\allowbreak 23,\allowbreak 29\}$ &  \\
\arrayrulecolor{black!25}\midrule[.25pt]\arrayrulecolor{black}
$G_{33}$ & $4,\allowbreak 6,\allowbreak 10,\allowbreak 12,\allowbreak 18$ & 18 & 3 & 1 & 6 & $\{1\}$ & \cref{c:thm:GG} \\*
 &  &  &  & 2 &  & $\{5\}$ &  \\
\arrayrulecolor{black!25}\midrule[.25pt]\arrayrulecolor{black}
$G_{34}$ & $6,\allowbreak 12,\allowbreak 18,\allowbreak 24,\allowbreak 30,\allowbreak 42$ & 42 & 3 & 1 & 42 & $\{1,\allowbreak 13,\allowbreak 19,\allowbreak 25,\allowbreak 31,\allowbreak 37\}$ & \cref{c:thm:GG} \\*
 &  &  &  & 2 &  & $\{5,\allowbreak 11,\allowbreak 17,\allowbreak 23,\allowbreak 29,\allowbreak 41\}$ &  \\
\arrayrulecolor{black!25}\midrule[.25pt]\arrayrulecolor{black}
$G_{35}$ & $2,\allowbreak 5,\allowbreak 6,\allowbreak 8,\allowbreak 9,\allowbreak 12$ & 12 & 1 & 1 & 6 & $\{1,\allowbreak 5\}$ & \cref{c:thm:GG} \\
\arrayrulecolor{black!25}\midrule[.25pt]\arrayrulecolor{black}
$G_{36}$ & $2,\allowbreak 6,\allowbreak 8,\allowbreak 10,\allowbreak 12,\allowbreak 14,\allowbreak 18$ & 18 & 1 & 1 & 6 & $\{1,\allowbreak 5\}$ & \cref{c:thm:GG} \\
\arrayrulecolor{black!25}\midrule[.25pt]\arrayrulecolor{black}
$G_{37}$ & $2,\allowbreak 8,\allowbreak 12,\allowbreak 14,\allowbreak 18,\allowbreak 20,\allowbreak 24,\allowbreak 30$ & 30 & 1 & 1 & 30 & $\{1,\allowbreak 7,\allowbreak 11,\allowbreak 13,\allowbreak 17,\allowbreak 19,\allowbreak 23,\allowbreak 29\}$ & \cref{c:thm:GG} \\
\bottomrule
\caption{Exceptional posh pairs in \cref{c:thm:main}, with $U=(V^*)^{\sigma_k}$. The Coxeter number $h=d_r$ is displayed precisely for the well-generated groups. Only $k=1$ gives a permutation module.}\label{c:tab:exceptional}\\
\end{longtable}
\endgroup
\endgroup
\begingroup\small
\begin{table}[htbp]
\centering\small
\setlength{\tabcolsep}{4pt}
\begin{tabular}{@{}lL{.31\textwidth}rrL{.25\textwidth}@{}}
\toprule $W$ & Degrees & $h$ & Period $M$ & $p\bmod M$\\\midrule
$G_{4}$ & $4,\allowbreak 6$ & 6 & 6 & $\{1,\allowbreak 3\}$ \\
$G_{5}$ & $6,\allowbreak 12$ & 12 & 6 & $\{1\}$ \\
$G_{6}$ & $4,\allowbreak 12$ & 12 & 12 & $\{1,\allowbreak 9\}$ \\
$G_{7}$ & $12,\allowbreak 12$ & -- & 12 & $\{1\}$ \\
$G_{8}$ & $8,\allowbreak 12$ & 12 & 12 & $\{1,\allowbreak 5\}$ \\
$G_{9}$ & $8,\allowbreak 24$ & 24 & 24 & $\{1,\allowbreak 17\}$ \\
$G_{10}$ & $12,\allowbreak 24$ & 24 & 12 & $\{1\}$ \\
$G_{11}$ & $24,\allowbreak 24$ & -- & 24 & $\{1\}$ \\
$G_{12}$ & $6,\allowbreak 8$ & -- & 24 & $\{1,\allowbreak 11,\allowbreak 17,\allowbreak 19\}$ \\
$G_{13}$ & $8,\allowbreak 12$ & -- & 24 & $\{1,\allowbreak 17\}$ \\
$G_{14}$ & $6,\allowbreak 24$ & 24 & 24 & $\{1,\allowbreak 19\}$ \\
$G_{15}$ & $12,\allowbreak 24$ & -- & 24 & $\{1\}$ \\
$G_{16}$ & $20,\allowbreak 30$ & 30 & 30 & $\{1,\allowbreak 11\}$ \\
$G_{17}$ & $20,\allowbreak 60$ & 60 & 60 & $\{1,\allowbreak 41\}$ \\
$G_{18}$ & $30,\allowbreak 60$ & 60 & 30 & $\{1\}$ \\
$G_{19}$ & $60,\allowbreak 60$ & -- & 60 & $\{1\}$ \\
$G_{20}$ & $12,\allowbreak 30$ & 30 & 30 & $\{1,\allowbreak 19\}$ \\
$G_{21}$ & $12,\allowbreak 60$ & 60 & 60 & $\{1,\allowbreak 49\}$ \\
$G_{22}$ & $12,\allowbreak 20$ & -- & 60 & $\{1,\allowbreak 29,\allowbreak 41,\allowbreak 49\}$ \\
$H_3=G_{23}$ & $2,\allowbreak 6,\allowbreak 10$ & 10 & 10 & $\{1,\allowbreak 5,\allowbreak 9\}$ \\
$G_{24}$ & $4,\allowbreak 6,\allowbreak 14$ & 14 & 14 & $\{1,\allowbreak 9,\allowbreak 11\}$ \\
$G_{25}$ & $6,\allowbreak 9,\allowbreak 12$ & 12 & 6 & $\{1\}$ \\
$G_{26}$ & $6,\allowbreak 12,\allowbreak 18$ & 18 & 6 & $\{1\}$ \\
$G_{27}$ & $6,\allowbreak 12,\allowbreak 30$ & 30 & 30 & $\{1,\allowbreak 19,\allowbreak 25\}$ \\
$F_4=G_{28}$ & $2,\allowbreak 6,\allowbreak 8,\allowbreak 12$ & 12 & 6 & $\{1,\allowbreak 5\}$ \\
$G_{29}$ & $4,\allowbreak 8,\allowbreak 12,\allowbreak 20$ & 20 & 20 & $\{1,\allowbreak 9,\allowbreak 13,\allowbreak 17\}$ \\
$H_4=G_{30}$ & $2,\allowbreak 12,\allowbreak 20,\allowbreak 30$ & 30 & 30 & $\{1,\allowbreak 11,\allowbreak 19,\allowbreak 29\}$ \\
$G_{31}$ & $8,\allowbreak 12,\allowbreak 20,\allowbreak 24$ & -- & 60 & $\{1,\allowbreak 13,\allowbreak 17,\allowbreak 29,\allowbreak 37,\allowbreak 41,\allowbreak 49,\allowbreak 53\}$ \\
$G_{32}$ & $12,\allowbreak 18,\allowbreak 24,\allowbreak 30$ & 30 & 30 & $\{1,\allowbreak 7,\allowbreak 13,\allowbreak 19\}$ \\
$G_{33}$ & $4,\allowbreak 6,\allowbreak 10,\allowbreak 12,\allowbreak 18$ & 18 & 6 & $\{1\}$ \\
$G_{34}$ & $6,\allowbreak 12,\allowbreak 18,\allowbreak 24,\allowbreak 30,\allowbreak 42$ & 42 & 42 & $\{1,\allowbreak 13,\allowbreak 19,\allowbreak 25,\allowbreak 31,\allowbreak 37\}$ \\
$E_6=G_{35}$ & $2,\allowbreak 5,\allowbreak 6,\allowbreak 8,\allowbreak 9,\allowbreak 12$ & 12 & 6 & $\{1,\allowbreak 5\}$ \\
$E_7=G_{36}$ & $2,\allowbreak 6,\allowbreak 8,\allowbreak 10,\allowbreak 12,\allowbreak 14,\allowbreak 18$ & 18 & 6 & $\{1,\allowbreak 5\}$ \\
$E_8=G_{37}$ & $2,\allowbreak 8,\allowbreak 12,\allowbreak 14,\allowbreak 18,\allowbreak 20,\allowbreak 24,\allowbreak 30$ & 30 & 30 & $\{1,\allowbreak 7,\allowbreak 11,\allowbreak 13,\allowbreak 17,\allowbreak 19,\allowbreak 23,\allowbreak 29\}$ \\
\bottomrule
\end{tabular}
\caption{Untwisted exceptional posh pairs $(p,V^*)$ in \cref{c:thm:untwisted}. These are the $k=1$ rows of \cref{c:tab:exceptional}; the Coxeter number $h=d_r$ is displayed precisely for the well-generated groups. The degree conditions agree with~\cite[Table 1]{IO}.}\label{tab:exceptional-untwisted}
\end{table}
\endgroup

\clearpage
\begingroup\raggedbottom

\endgroup
\end{document}